\documentclass[12pt,a4paper]{extarticle}
\usepackage[english]{babel}
\usepackage[letterpaper,top=2cm,bottom=2cm,left=2.5cm,right=2.5cm,marginparwidth=1.75cm]{geometry}
\usepackage{amsmath, amsfonts, amsthm}
\usepackage{titlesec}
\usepackage{abstract}

\usepackage{tikz-cd}

\definecolor{darkblue}{rgb}{0.15,0.35,0.55}
\definecolor{reddish}{rgb}{0.65, 0.2, 0.2}
\usepackage{graphicx}
\usepackage{hyperref}
\hypersetup{
colorlinks=true,
citecolor=darkblue,
linkcolor=reddish,
urlcolor=darkblue,
pdfauthor={},
pdftitle={},
pdfsubject={}
}
\usepackage{mathtools}
\usepackage{amssymb} 
\usepackage[backend=bibtex]{biblatex}
\allowdisplaybreaks 
\usepackage{comment} 
\usepackage[capitalize]{cleveref}
\usepackage{bbm}
\usepackage{nicematrix}
\usepackage{tikz}
\usepackage[ruled,linesnumbered]{algorithm2e}

\newtheorem{theorem}{Theorem}[section]
\newtheorem{lemma}[theorem]{Lemma}
\newtheorem{corollary}[theorem]{Corollary}
\newtheorem{proposition}[theorem]{Proposition}
\newtheorem{conjecture}[theorem]{Conjecture}
\theoremstyle{definition}
\newtheorem{definition}[theorem]{Definition}
\newtheorem{remark}[theorem]{Remark}

\newtheorem{example}[theorem]{Example}

\newcommand{\R}{\mathbb{R}}
\newcommand{\C}{\mathbb{C}}
\newcommand{\Z}{\mathbb{Z}}
\newcommand{\K}{\mathbb{K}}

\title{\bf Constraining Conformal Correlators}
\author{Viktoriia Borovik, Claire de Korte, Nathan Meurrens, Dmitrii Pavlov}
\date{}

\begin{document}
\maketitle
\begin{abstract}
We study the space of 
conformally covariant 
$n$-point functions of spinning ope\-rators using methods from invariant theory, commutative algebra, and combinatorics. We show that the rational part of any such function can be expressed in terms of the basic building blocks introduced by Costa, Penedones, Poland, Rychkov, thereby providing  a rigorous proof of a result 
that is widely used in the physics literature. We 
reformulate the problem of enumeration of 
$n$-point structures in terms of counting lattice points in fractional matching polytopes, and compute these counts using vector partition functions, Hilbert functions, and Kostka numbers. 
We show that all algebraic relations between the building blocks follow from Gram constraints and compute the number of algebraically independent building blocks. 
For three-point functions, we derive closed counting formulas for arbitrary integer spins, both with and without  Bose symmetry, and discuss a necessary and sufficient condition for the partial conservation operator to lift to a differential operator written in terms of the building blocks. We provide code that generates a basis of three-point structures satisfying these constraints for given values of spins and scaling dimensions.
\end{abstract}

\begingroup
\hypersetup{linkcolor=black}
\setcounter{tocdepth}{1}
\tableofcontents
\endgroup

\newpage

\section{Introduction}

Conformal field theories (CFTs) are models of physical systems that are invariant under conformal transformations. Given a conformal field theory, its conformal correlators are correlation functions that encode the likelihood of interactions among particles as a function of their positions $\mathbf{x}\in\mathbb{R}^d$. 
Conformal symmetry requires correlators to transform covariantly under the relevant representation of the conformal group \cite[§III.C]{PRV}. Throughout this paper, we will be considering a specific subset of representations known to physicists as corresponding to conformal primaries \cite[§III.B]{PRV}. For bosonic operators, the relevant representations are tensor representations of the conformal group. Direct computation of these correlators can be extremely complicated. An alternative approach is to `bootstrap' them, using the constraints imposed by conformal symmetry to restrict the form of a correlator and, in some cases, to determine it completely; see \cite{SimmonsDuffin} for an introduction to  conformal bootstrap. Since conformal transformations are nonlinear, deriving these constraints is nontrivial. 
The embedding space formalism simplifies this dramatically, by realizing conformal transformations and their tensor representations in $\mathbb{R}^d$ as restrictions to the Poincar\'e section of Lorentz transformations and their tensor representations in $\mathbb{R}^{1,d+1}$; for details we refer to \cite{CPPR,Dirac,FGG2,FGG,MackSalam,Weinberg}.

\medskip

In this paper, we work  primarily within the embedding space formalism. Using the constraints imposed by conformal symmetry, together with Bose symmetry and partial conservation, we characterize the space of allowed structures in bosonic conformal correlators. 
We provide explicit procedures for computing these spaces for correlators of any number of ope\-rators with arbitrary fixed scaling dimensions and integer spins. Related problems have previously been studied, though from different perspectives and in more restricted settings; see, for example, \cite{DanielNathan,CT,CPPR,DKKPS, GHL, HLMM, KSD,Zhiboedov}.

\medskip
From the algebraic perspective, our starting point is the collection of \emph{basic building blocks} introduced in \cite{CPPR}. These are rational functions  defined as follows:
\begin{equation}\label{eq: confblocks intro}
    \begin{aligned}
        P_{ij} & =P_i\cdot P_j, \\ 
H_{ij} & =-2\big[(Z_i\cdot Z_j)(P_i\cdot P_j)
-(Z_i\cdot P_j)(Z_j\cdot P_i)\big],\\
V_{i,jk}
& =
\frac{(Z_i\cdot P_j)(P_i\cdot P_k)
-(Z_i\cdot P_k)(P_i\cdot P_j)}
{P_j\cdot P_k}.
    \end{aligned}
\end{equation}
Here $P_i$ and $Z_i$ are null embedding-space vectors representing position and polarisation respectively. In the physics literature, these basic building blocks are used as generators of conformally covariant structures, and it is reasoned that every bosonic conformal correlator can be expressed in terms of them. Mathematically, this means that the building blocks generate the field of rational invariants under the action of the Lorentz group, together with another group action that we call transversality. In Theorem~\ref{thm: main}, we give a rigorous proof of this statement, which is standard in the conformal bootstrap literature but, to our knowledge, had not yet been established in this form.
\medskip

Since the building blocks are not generally algebraically independent, it is natural to study the algebraic varieties they parametrize and their defining ideals.
This problem was recently investigated in \cite{MSS}, where the authors studied these ideals as well as Gram constraints arising from the vectors $P_i$ and $Z_i$.
In Section \ref{sec: alg rel}, we take this as a point of departure, but with a different emphasis. We consider the building blocks $P_{ij}$, $H_{ij}$ and $\mathcal{V}_{ij}:=V_{i,i+1j}$ and show that any algebraic relation between them is implied by the Gram constraints satisfied by the inner products $P_i\cdot P_j$, $P_i\cdot Z_j$ and $Z_i\cdot Z_j$.
We also show that any relation between the building blocks involving $V_{i,jk}$ can be rewritten using only the subset of the blocks of the form $\mathcal{V}_{ij}$. 
We then compute the number of building blocks $H_{ij}$ and $\mathcal{V}_{ij}$ that are algebraically independent over the field of rational functions in the cross-ratios in the blocks $P_{ij}$.

\medskip

Once conformally covariant structures are expressed in terms of the building blocks~\eqref{eq: confblocks intro} subject to the appropriate homogeneity conditions~\eqref{eq: homogeneity preliminaries}, one would like to count how many such structures exist for a set of prescribed spins. This counting problem appears frequently in the physics literature, especially for three- and four-point correlators, but explicit formulas are usually derived case by case or, as introduced in \cite{KSD}, through Clebsch-Gordan decompositions. In Theorem~\ref{thm: main sec 3}, we show that this problem admits a reformulation in terms of multigraded Hilbert functions, vector partition functions, and lattice-point counts in fractional matching polytopes. In Section~\ref{sec: alg rel}, we revisit this problem while accounting for algebraic dependencies among the building blocks, and compute the corresponding Hilbert~function.

\medskip

Finally, we give a detailed analysis of three-point correlation functions. In Propositions~\ref{prop:Bose_threespins} and~\ref{prop:Bose-two-spins}, we derive closed-form formulas for the number of conformally covariant structures with arbitrary spins after imposing Bose symmetry. We also study the partial conservation ope\-rator in Section~\ref{sec: conservation}. In Theorem~\ref{thm: diff op 1}, we provide necessary and sufficient conditions for the action of a differential operator $\partial_{P_i}\cdot D_{Z_i}$ on three-point functions to be captured by a differential operator $\mathcal{D}_{i,1}$ acting on the numerator in terms of the building blocks, and construct this operator explicitly, reproducing the result of \cite{Zhiboedov}. Moreover, we provide an algorithm to explicitly construct, for any choice of $s_i-t_i$, the differential operator $\mathcal{D}_{i,s_i-t_i}$ that captures the action of $(\partial_{P_i}\cdot D_{Z_i})^{s_i-t_i}$ in terms of the building blocks. We apply this to explicitly construct $\mathcal{D}_{i,2}$ for $d=3$. Necessary and sufficient conditions for the existence of higher-order operators $(\partial_{P_i}\cdot D_{Z_i})^{s_i-t_i}$ are the content of Conjectures~\ref{conj: higher order derivative} and~\ref{conj: 3point diff op conjecture}.

\medskip

We also provide code that computes a basis of three-point structures, expressed in terms of the building blocks, satisfying all of the aforementioned constraints. This code, along with other auxiliary files for our paper, is available at \cite{zenodo}:
\begin{center}
    \url{https://doi.org/10.5281/zenodo.20444515}.
\end{center}

\medskip

\noindent\textbf{Outline.} Our paper is organized as follows. In Section~\ref{sec: preliminaries} we review the embedding space formalism and the necessary mathematical background. Section~\ref{sec:matching_pols} develops the combinatorial and algebraic framework underlying the counting problem. In Section~\ref{sec:unipotent} we turn to invariant theory: we define the unipotent group action encoding transversality and prove that the basic building blocks generate the field of rational invariants. In Section~\ref{sec:ConfNpoint} we revisit the general form of conformally covariant $n$-point structures from \cite{CPPR} and relate their enumeration to the multigraded counting problem studied earlier.
Section~\ref{sec: alg rel} studies algebraic relations among the building blocks in low-dimensional cases and establishes results on the dimensions governing the number of algebraically independent functions. Section~\ref{sec: Bose} concerns Bose symmetry: we describe the action of permutation groups on the space of conformally covariant structures and compute the resulting dimension reduction in the case of three-point correlators. In Section~\ref{sec: conservation} we study partial conservation and the associated differential operators on spaces of conformally covariant structures.

\section{Mathematical and physical setup}\label{sec: preliminaries}
This section has two aims. First, we introduce the \textit{embedding space formalism} for \textit{bosonic conformal correlators}. Second, we review the relevant mathematical background. 
Since our purpose is to use rather than develop the embedding space formalism, we briefly recall only the aspects needed for this work and refer the reader to \cite{CPPR,Dirac,FGG2,FGG,MackSalam,Weinberg} for 
details.

\medskip

\subsection{Embedding space formalism}
Throughout this paper, $(\,\cdot\,)$ between bold characters denotes the Euclidean scalar product, while $(\,\cdot\,)$ between plain characters denotes the Lorentzian scalar product. We will be working in $d\ge 3$ spatial dimensions. Bold vectors $\mathbf{x}_i\in\mathbb{R}^d$ are used for the spatial position of the $i$th particle, and plain vectors $P_i\in\mathbb{R}^{1,d+1}$ for their embedding-space counterparts, the latter transforming in the fundamental representation of the Lorentz group ${\rm O}(1,d+1)$. Restricting to the \emph{Poincar\'e section}, defined as 
\begin{equation}\label{eq: Poincare}
P_i=(P_i^0,\mathbf{P}_i,P_i^{d+1})=\left(\tfrac{1}{2}[1+\mathbf{x}_i\cdot\mathbf{x}_i],\,\mathbf{x}_i,\,\tfrac{1}{2}[1-\mathbf{x}_i\cdot\mathbf{x}_i]\right),    
\end{equation}
induces the fundamental action of the conformal group $\mathrm{Conf}(d)$ on $\mathbf{x}_i\in\mathbb{R}^d$. See~\cite{Dirac} for the original construction, and \cite[Sections~4.1–4.2]{FMS} for an introduction to the conformal group and its representations. As explained in \cite{CPPR,FGG2,FGG,MackSalam,Weinberg}, this method of realizing conformal transformations as restrictions of Lorentz transformations extends beyond the fundamental representations of $\mathrm{Conf}(d)$ and ${\rm O}(1,d+1)$ to their tensor representations. 
Thus, bosonic conformal correlators can be realized as restrictions to the Poincar\'e section~\eqref{eq: Poincare} of embedding-space correlators. This is useful for the conformal bootstrap because, in the embedding space, conformal covariance lifts to Lorentz covariance together with suitable homogeneity, symmetry, tracelessness and transversality conditions~\eqref{eq: Lorentz inv}–\eqref{eq: traceless}; see \cite{FGG}.

\medskip

In \cite[p.~8–13]{CPPR}, the embedding space formalism is simplified by using index-free notation, based on the correspondence between symmetric tensors of order $s$ and homogeneous polynomials of degree $s$.
Concretely, embedding-space correlators are contracted with \textit{polarisation vectors} $Z_i\in\mathbb{C}^{1,d+1}$, one for each particle, that transform under the fundamental representation of ${\rm O}(1,d+1)$. Contracting a symmetric set of $s_i$ indices of the correlator with $s_i$ copies of $Z_i$ turns the correlator into  a polynomial of degree $s_i$ in the $m=d+2$ entries of~$Z_i$; contracting all indices yields a Lorentz-invariant polynomial in the entries of the polarisation vectors. Follo\-wing \cite{CPPR}, symmetric traceless tensors are encoded by imposing $Z_i\cdot Z_i=0$, whilst transversality is encoded by additionally imposing~$P_i\cdot Z_i=0$ and invariance under \[Z_i\mapsto Z_i+\alpha_iP_i.\]
Upon restriction to the Poincar\'e section, one recovers the polarisation
vectors of the conformal correlator, namely the vectors
$\mathbf{z}_i\in\mathbb{C}^d$ satisfying  $\mathbf{z}_i\cdot\mathbf{z}_i=0$. This is done by
setting
$$
Z_i=(Z_i^0,\mathbf{Z}_i,Z_i^{d+1})
=(\mathbf{x}_i\cdot\mathbf{z}_i,\mathbf{z}_i,-\mathbf{x}_i\cdot\mathbf{z}_i).
$$
On this section one automatically has
$
P_i\cdot P_i=P_i\cdot Z_i=Z_i\cdot Z_i=0
$.
Modulo the projective rescaling of $P_i$ and the gauge redundancy under 
$Z_i\mapsto Z_i+\alpha_iP_i$, the restriction from the space of null vectors $P_i$ and $Z_i$ satisfying $P_i\cdot Z_i=0$ to the Poincar\'e section is one-to-one.

\medskip

An \textit{$n$-point conformal correlator} (or an \emph{$n$-point function}) depends on $n$ position vectors
$\mathbf{x}_i\in\mathbb{R}^d$ for $i=1,\ldots,n$. In the index-free
embedding space formalism, bosonic conformal correlators are represented by functions
$$
G_{\mathbf{s},\mathbf{\Delta}}(P_1,Z_1,\ldots,P_n,Z_n),
$$
where, as introduced above, $P_i$ and $Z_i$ are the embedding-space vectors representing the position and polarisation of the $i$th particle. The tensor representation of ${\rm Conf}(d)$ that these functions correspond to is determined by  their spins
$\mathbf{s}=(s_1,\ldots,s_n)\in\mathbb{Z}_{\ge0}^n$ and scaling dimensions
$\mathbf{\Delta}=(\Delta_1,\ldots,\Delta_n)\in\mathbb{C}^n$. Below we
summarise the consequences of the discussion above, listing the conditions satisfied by functions $G_{\mathbf{s},\mathbf{\Delta}}(P_1,Z_1,\ldots,P_n,Z_n)$:
\begin{enumerate}
\itemsep0em
    \item \textit{Lorentz invariance.} For all $\Lambda\in{\rm O}(1,d+1)$, we have
    \begin{equation}\label{eq: Lorentz inv}
    G_{\mathbf{s},\mathbf{\Delta}}(\Lambda P_1,\Lambda Z_1,\ldots,\Lambda P_n,\Lambda Z_n)
    =
    G_{\mathbf{s},\mathbf{\Delta}}(P_1,Z_1,\ldots,P_n,Z_n).
    \end{equation}

    \item \textit{Transversality.} Equivalently, gauge invariance under
    $Z_i\mapsto Z_i+\alpha_iP_i$:
    \begin{equation}\label{eq: Transv inv}
    G_{\mathbf{s},\mathbf{\Delta}}(P_1,Z_1+\alpha_1P_1,\ldots,P_n,Z_n+\alpha_nP_n)
    =
    G_{\mathbf{s},\mathbf{\Delta}}(P_1,Z_1,\ldots,P_n,Z_n).
    \end{equation}

    \item \textit{Polynomiality and homogeneity.} The correlator is polynomial
    in the entries of $Z_i$ vectors and homogeneous of degree $-\Delta_i$ in $P_i$ and degree
    $s_i$ in $Z_i$, in other words,
    \begin{equation}\label{eq: homogeneity preliminaries}
        G_{\mathbf{s},\mathbf{\Delta}}(\lambda_1P_1,\mu_1Z_1,\ldots,\lambda_nP_n,\mu_nZ_n)
        =
        \left[\prod_{i=1}^n \lambda_i^{-\Delta_i}\mu_i^{s_i}\right]
        G_{\mathbf{s},\mathbf{\Delta}}(P_1,Z_1,\ldots,P_n,Z_n).
    \end{equation}

    \item \textit{Null-cone and polarisation constraints.} The correlator is
    considered on the locus
    \begin{equation}\label{eq: traceless}
       P_i\cdot P_i=P_i\cdot Z_i=Z_i\cdot Z_i=0,
       \qquad i=1,\ldots,n.
    \end{equation}
    Here $P_i^2=0$ is the
    null-cone condition and $Z_i^2=0$
    encodes tracelessness.
\end{enumerate}

Certain conformal correlators satisfy further conditions beyond conformal
covariance. In standard quantum field theory (QFT) derivations, 
these conditions, as the ones above, are built into the formalism and are therefore automatically satisfied. 
In the bootstrap approach, by contrast, they become additional constraints on the correlators.
Two important examples are \textit{Bose symmetry} and
\textit{partial conservation}, discussed in Sections~\ref{sec: Bose}
and~\ref{sec: conservation}, respectively.

\medskip

Bose symmetry applies to correlators involving identical bosonic particles. If
the particles $i$ and $j$ are identical, the correlator must be
invariant under exchanging $(P_i,Z_i)$ and $(P_j,Z_j)$:
\begin{equation} \label{eq:bose}
\begin{split}
    &G_{\mathbf{s},\mathbf{\Delta}}(P_1,Z_1,\ldots,P_i,Z_i,\ldots,P_j,Z_j,\ldots,P_n,Z_n)\\
    &\qquad =
    G_{\mathbf{s},\mathbf{\Delta}}(P_1,Z_1,\ldots,P_j,Z_j,\ldots,P_i,Z_i,\ldots,P_n,Z_n).
\end{split}
\end{equation}
Partial conservation applies if, for some $i$, one has $s_i\geq 1$ and
$\Delta_i=d-1+t_i$, where $t_i\in\{0,\ldots,s_i-1\}$ denotes the \textit{depth} 
of the $i$th particle. Then, by \cite{DNW}, $G_{\mathbf{s},\mathbf{\Delta}}$
must satisfy
\begin{equation}\label{eq: conservation diff eq}
    \left(\frac{\partial}{\partial P_i}\cdot D_{Z_i}\right)^{s_i-t_i}
    G_{\mathbf{s},\mathbf{\Delta}}(P_1,Z_1,\ldots,P_n,Z_n)=0,
\end{equation}
where $D_Z$ denotes  the \emph{Thomas–Todorov differential operator} defined in \cite[1.30a]{DPPT}:
\begin{equation}
    D_{Z}
    =
    \left(\frac{d}{2}-1+Z\cdot\frac{\partial}{\partial Z}\right)
    \frac{\partial}{\partial Z}
    -\frac{1}{2}Z
    \frac{\partial^2}{\partial Z\cdot\partial Z}.
\end{equation}
Here $\frac{\partial}{\partial Z}$ denotes the gradient with respect to the vector $Z$, so that contractions such as
$Z\cdot\frac{\partial}{\partial Z}$ are taken using the Lorentzian
metric. The equation~\eqref{eq: conservation diff eq} is understood on the constraint locus
$P_i^2=P_i\cdot Z_i=Z_i^2=0$, equivalently modulo the ideal generated by these
constraints.

\medskip

Before moving on, we make a brief comment on parity. Functions that are invariant under ${\rm O}(1,d+1)$ are called parity even, whilst those that are only invariant under the subgroup $\mathrm{SO}(1,d+1)$ are called parity odd. By requiring our correlators to be invariant under ${\rm O}(1,d+1)$, we have in particular imposed them to be parity even. In practise, parity odd terms can be introduced as contractions of embedding-space vectors $P_i$ and $Z_i$ with the $(d+2)$-dimensional Levi-Civita tensor. Here we will focus on parity even structures, but our results can easily be adapted to account for parity odd structures using \cite[Section 4.2.3]{CPPR}.

\subsection{Mathematical background}
Having outlined the embedding space formalism, we  recall the main
mathematical notions used in this paper. For more details, see~{\cite{cox1997ideals, michalekSturmfels2021, CCA, Shafarevich2}.}

\medskip

Let $\mathbf{x}=(x_1,\ldots,x_n)$ and
$R=\mathbb{K}[\mathbf{x}]=\mathbb{K}[x_1,\ldots,x_n]$ be the polynomial ring in $n$ variables over a field $\K$. A \emph{multigrading} of $R$
by an abelian group $A$ is a semigroup homomorphism
$$
    \deg:\mathbb{N}^n\to A,
$$
which assigns to each monomial
$\mathbf{x}^{\mathbf{u}}=x_1^{u_1}\cdots x_n^{u_n}$ the multidegree
$\deg(\mathbf{u})$.
A polynomial is \emph{homogeneous} of multidegree $\mathbf{a}\in A$ if all
its monomials have degree $\mathbf{a}$. The grading implies: 
\begin{equation}
    R=\bigoplus_{\mathbf{a}\in A} R_{\mathbf{a}},
    \qquad
    R_{\mathbf{a}}R_{\mathbf{b}}\subseteq R_{\mathbf{a}+\mathbf{b}},    
\end{equation}
where $R_{\mathbf{a}}$ is the space of homogeneous polynomials of
degree $\mathbf{a}$.
In this paper, $A=\mathbb{Z}^m$ for some $m$, and the multigrading is encoded by an integer matrix $B\in \mathbb Z^{m\times n}$. If ${\bf u}\in \mathbb Z_{\geq 0}^n$ is the exponent vector of a monomial, then its multidegree is
$
\deg({\bf x^u})=B \cdot {\bf u}
$.
For a multigraded ring~$R$, the \emph{multigraded Hilbert function}
and \emph{Hilbert series} are defined respectively as
\begin{equation}\label{eq: HS preliminaries}
 h_R:\mathbb{Z}^m\to\mathbb{Z}_{\geq 0},
    \quad
    h_R(\mathbf{a})  =\dim_{\mathbb{K}} R_{\mathbf{a}}, \qquad \text{and} \qquad 
 {\rm HS}_R(\mathbf{t})
     =
    \sum_{\mathbf{a}\in\mathbb{Z}^m}
    h_R(\mathbf{a})\mathbf{t}^{\mathbf{a}}.   
\end{equation}
In our case, $ {\rm HS}_R(\mathbf{t})$ will always be a formal power series, which we compute via
\cite[Lem.~8.16]{CCA}.

\medskip

We also use basic notions from algebraic geometry. Let $\mathbb{K}$ be an
algebraically closed field of characteristic zero. An \emph{affine variety} in
$\mathbb{K}^n$ is the common zero set of finitely many polynomials in
$R = \mathbb{K}[x_1,\ldots,x_n]$, while a \emph{projective variety} in
$\mathbb{P}_{\mathbb{K}}^n$ is the common zero set of finitely many
homogeneous polynomials in $\mathbb{K}[x_0,\ldots,x_n]$. A variety is
\emph{irreducible} if it is not the union of two proper subvarieties.
Polynomials vanishing on a variety form an
\emph{ideal} $I \subseteq R$, which is a subset of a ring $R$, closed under addition and
under multiplication by elements of~$R$. The ideal $I$ is \emph{prime} if $ab\in I$
implies $a\in I$ or $b\in I$. Any algebraic variety $X$ is completely
determined by its vanishing ideal $\mathcal{I}(X)$, and irreducible varieties correspond to prime ideals. 
Equivalently, an affine variety $X$ defined by the ideal $\mathcal{I}(X)$ can be identified with the \emph{spectrum}  $\mathrm{Spec}(\mathcal{R})$ (that is, the set of all prime ideals) of the \emph{quotient ring} $\mathcal{R}:=\mathbb{K}[\mathbf{x}]/\mathcal{I}(X)$, whose elements are given by cosets of the form $a + \mathcal{I}(X)$ with $a \in R=\K[x_1,\dots,x_n]$.

\medskip

The spaces $\mathbb{K}^n$ and $\mathbb{P}^n_{\mathbb{K}}$ can be endowed with the \emph{Zariski topology}, in which closed sets
are exactly algebraic varieties. This induces the Zariski topology on every subvariety in $\mathbb{K}^n$ and~$\mathbb{P}^n_{\mathbb{K}}$. The \emph{Zariski closure} of a subset in $\mathbb{K}^n$ and $\mathbb{P}^n_{\mathbb{K}}$ is then the
smallest variety containing~it.
\begin{remark}\label{rem: functions as formal vars}
    In this paper, algebraic varieties will  be described parametrically. We then pass to quotient rings as follows: if functions $f_1({\bf x}),\ldots,f_k({\bf x})$ define a  map to $\mathbb \K^k$, we introduce formal variables $y_1,\ldots,y_k$ and quotient $\mathbb K[y_1,\ldots,y_k]$ by the kernel of the homomorphism
$$
\mathbb K[y_1,\ldots,y_k]\rightarrow \mathbb K[\mathbf x],
\qquad
y_i\mapsto f_i(\mathbf{x}).
$$
This quotient is the \emph{coordinate ring} of the Zariski closure of the parametrized image. In our setup the role of the $y$ variables is played by the building blocks $P_{ij}$, $H_{ij}$, and $V_{i,jk}$, and the role of the $x$ variables is taken by the entries of the vectors $P_i$ and $Z_i$. We will sometimes treat $P_{ij}$, $H_{ij}$, and $V_{i,jk}$ as formal variables, which means that we ignore this quotient operation.
\end{remark}
In this work we are interested in the number of \emph{algebraically independent} conformally covariant $n$-point structures in the following sense. Elements
$f_1,\ldots,f_n$ of a commutative $\mathbb{K}$-algebra are
\emph{algebraically dependent} over an arbitrary field $\mathbb{K}$ if there exists a
polynomial $p$ in $\mathbb{K}[x_1,\ldots,x_n]\setminus\{0\}$ such that
$p(f_1,\ldots,f_n)=0$; otherwise they are \emph{algebraically independent}.
\begin{definition}[Transcendence degree] \label{def: trdeg}
    Let $\mathbb{K} \subset \mathbb{L}$ be a field extension. 
    A maximal algebraically
independent subset of $\mathbb{L}$ over $\mathbb{K}$ is a
\emph{transcendence basis}, and its cardinality is the
\emph{transcendence degree}, denoted
$\mathrm{trdeg}_{\mathbb{K}}\mathbb{L}$. This number does not depend on the choice of basis. 
\end{definition}

Since the 
building blocks of \cite{CPPR} are invariants of group
actions, we recall the relevant~invariant theory terminology. Here, we will be interested in actions of \emph{algebraic groups}, that is, groups that themselves carry the structure of an algebraic variety, on algebraic varieties.
A left action of a group $G$ on a set $S$
is a map $G\times S\to S$, written $(g,x)\mapsto g\cdot x$, such that
$e\cdot x=x$ and $g\cdot(h\cdot x)=(gh)\cdot x$. The orbit of $x\in S$ is
$\{g\cdot x:g\in G\}$. The~action is \emph{faithful} if only the identity acts
trivially on all of $S$, and \emph{free} if only the identity fixes any~point.

\medskip

In Section~\ref{sec:matching_pols} we relate constraints on
conformal correlators to lattice-point counting of polytopes, which is one of the central problems in discrete geometry.  A
\emph{polytope} in~$\mathbb{R}^n$ is the convex hull of finitely many points,
equivalently a bounded intersection of finitely many half-spaces
\cite[Sec.~1.1]{Ziegler}. A \emph{lattice point} is a point with integer
coordinates. If $P\subset\mathbb{R}^n$ is a full-dimensional lattice polytope,
then its \emph{Ehrhart polynomial} counts lattice points in $tP$:
\begin{equation}\label{eq: Ehrhart preliminaries}
 L(P,t)=\#(tP\cap\mathbb{Z}^n),
    \quad t\in\mathbb{Z}_{\geq 0}.   
\end{equation}
Ehrhart's theorem states that this function is indeed a polynomial in $t$
\cite[Theorem~12.2]{CCA}.

\section{Counting monomials}
\label{sec:matching_pols}

This section is dedicated purely to commutative algebra and combinatorics. In Section \ref{sec:ConfNpoint} its results will be related to counting conformally covariant 
$n$-point structures. We discuss the problem from three perspectives: algebraic~\eqref{per: alg}, geometric~\eqref{per: geo}, and combinatorial~\eqref{per: com}.

\medskip

Let $\K$ be a field of characteristic zero, and denote $[n]\coloneqq\{1,2,\ldots,n\}$. 
We fix a vector of non-negative integers ${\bf s}\coloneqq (s_1,\ldots,s_n)$ and consider 
$n(n-2)+\binom{n}{2}$ variables
\begin{align*}
\mathcal{V}_{ij} \quad & \text{for } i \in [n] \text{ and } 
j \in [n] \setminus \{i,i+1\}, \\
H_{ij} \quad & \text{for } 1 \le i < j \le n.
\end{align*}
We wish to count the number $N({\bf s})$ of polynomials
$
Q_{\bf s}(\mathcal{V}_{ij}, H_{ij})
$ with coefficients in $\K$
such~that
\begin{equation} \label{eq: hom_condition sec 3}
    Q_{\bf s}(\alpha_i \mathcal{V}_{ij}, \alpha_i \alpha_jH_{ij}) = \alpha_1^{s_1}\alpha_2^{s_2} \cdots \alpha_n^{s_n} \cdot Q_{\bf s}(\mathcal{V}_{ij}, H_{ij}) \quad \forall\, \alpha \in\K^n.
\end{equation}
We begin by stating the main outcome of this section. Subsections~\ref{per: alg}–\ref{per: com} present a detailed explanation of the two theorems below and the notions appearing in their statements.

\begin{theorem}\label{thm: main sec 3}
Consider polynomial rings
 $Q_n \coloneqq \K[\mathcal{V}_{ij}, H_{ij}]$ and $R_n \coloneqq \K[H_{ij}]$
endowed with $\mathbb{Z}^n$‑gradings determined by integer matrices $B$~\eqref{eq:matB} and $A$~\eqref{eq:matA}, respectively.
Let $h_R({\bf s})$ denote the Hilbert function of a $\mathbb{Z}^n$‑graded ring $R$ evaluated at
${\bf s}\in\mathbb{Z}_{\ge0}^n$.
Then the following coincide:
\begin{enumerate}
\item[(1)]
$N({\bf s})$;

\item[(2)]
the Hilbert function $h_{Q_n}({\bf s})$;

\item[(3)] the coefficient of $t_1^{s_1}t_2^{s_2}\cdots t_n^{s_n}$ in the Taylor expansion about the origin of the function
\begin{equation*}
\dfrac{1}{\prod\limits_{1\leq i\leq n}(1-t_i)^{n-2} \prod\limits_{1\leq i<j\leq n}(1-t_it_j)};    
\end{equation*}

\item[(4)]
the number of lattice points in the $n$‑point conformal polytope $\mathcal{C}_n({\bf s})$ defined in \eqref{eq: conformal polytope};

\item[(5)]
the vector partition function  
$
\varphi_B({\bf s})
$ defined in~\eqref{eq: defVectorPart},
that is, the number of ways to write the vector ${\bf s}$ as a non‑negative integer linear combination of the columns of $B$;

\item[(6)]
$$
\sum_{{\bf b}\le{\bf s}}
h_{R_n}({\bf b})
\prod_{i=1}^n
\binom{n-3+s_i-b_i}{s_i-b_i},
$$
where ${\bf b}\le{\bf s}$ is the coordinate-wise comparison.
\end{enumerate}
\end{theorem}
\noindent
By point (6) in Theorem~\ref{thm: main sec 3}, the Hilbert function $h_{Q_n}({\bf s})$ is expressed
in terms of the Hilbert function $h_{R_n}({\bf b})$. Thus, the analysis of
$h_{Q_n}$ reduces to the study of the ring $R_n=\K[H_{ij}]$.
\begin{theorem}\label{thm: main Rn sec 3}
Let $R_n=\K[H_{ij}]$ be endowed with the $\mathbb{Z}^n$‑grading defined by
the matrix $A$~\eqref{eq:matA}.
For any ${\bf b}\in\mathbb{Z}_{\ge0}^n$, the Hilbert function of $R_n$ satisfies
$$
h_{R_n}({\bf b})
= \varphi_A({\bf b}) =
\bigl|\mathcal{P}(K_n,{\bf b})\cap\mathbb{Z}^{\binom{n}{2}}\bigr|,
$$
where $\varphi_A$ is the vector partition function associated with the matrix $A$ and $\mathcal{P}(K_n,{\bf b})$ denotes the fractional perfect ${\bf b}$‑matching polytope~\eqref{eq: fractional perfect s-matching}
of the complete graph $K_n$.
Moreover, we get
\begin{equation}
    h_{R_n}(\mathbf{b}) = \sum_{\lambda \in \mathcal{P}'_{\mathrm{even}}} K_{\lambda, \mathbf{b}},
\end{equation}
where $\mathcal{P}'_{\mathrm{even}}$ denotes the set of partitions $\lambda$ of $|{\bf b}| = b_1 + \ldots + b_n$ whose
conjugate partition $\lambda'$ is even, and $K_{\lambda,{\bf b}}$ denotes the corresponding Kostka number.
\end{theorem}
\subsection{An algebraic perspective: multivariate Hilbert functions}\label{per: alg}
The generating function appearing in point (3) of Theorem~\ref{thm: main sec 3}
has a direct algebraic interpretation as the Hilbert series~\eqref{eq: HS preliminaries} of a multigraded polynomial ring.
Indeed, formula~\eqref{eq: hom_condition sec 3} means that the polynomials
$Q_{\bf s}(\mathcal{V}_{ij}, H_{ij})$ are $(s_1,\ldots,s_n)$‑homogeneous with respect to the
$\mathbb{Z}^n$‑grading on the polynomial ring
$
Q_n = \K[\mathcal{V}_{ij}, H_{ij}],
$
given by
\begin{equation}\label{eq: multidegrees VH}
\deg(\mathcal{V}_{ij}) = {\bf e}_i,
\qquad
\deg(H_{ij}) = {\bf e}_i + {\bf e}_j.
\end{equation}
Here, ${\bf e}_i$ denotes the $i$th standard basis vector of $\mathbb{Z}^n$.
We encode this multigrading in an
$n \times \bigl(n(n-2)+\binom{n}{2}\bigr)$ integer matrix $B$,
whose columns are indexed by the ring generators $\mathcal{V}_{ij}$ and $H_{ij}$ and are given by
the multidegree vectors in~\eqref{eq: multidegrees VH}:
\begin{equation}\label{eq:matB}
B = (b_{k,*}), \quad \text{where } \;
b_{k,\mathcal{V}_{ij}} =
\begin{cases}
1, & k=i,\\
0, & \text{otherwise},
\end{cases}
\quad
b_{k,H_{ij}} =
\begin{cases}
1, & k=i \text{ or } k=j,\\
0, & \text{otherwise}.
\end{cases}
\end{equation}
For a fixed integer vector ${\bf s}\in\mathbb{Z}_{\ge0}^n$, the set of homogeneous polynomials of multidegree ${\bf s}$
is a finite‑dimensio\-nal $\K$‑vector space, which we denote by $(Q_n)_{\bf s}$. We define
$
N({\bf s}) \coloneqq \dim_\K (Q_n)_{\bf s}.
$
By definition, $N({\bf s})$ is the value of the Hilbert function $h_{Q_n}({\bf s})$~\eqref{eq: HS preliminaries} of the graded ring~$Q_n$.
\begin{lemma}[{\cite[Lem.~8.16]{CCA}}] \label{lem:MHS}
The Hilbert series $\sum  h_{Q_n}({\bf s})\, {\bf t}^{\bf s}$ of the ring $Q_n$ is given by
\begin{equation}\label{eq: HS conformal}
    \dfrac{1}{\prod\limits_{i=1}^n \,\prod\limits_{j\in[n]\setminus\{i,i+1\}} (1-\mathbf{t}^{\deg \mathcal{V}_{ij}}) \prod\limits_{1\leq i<j\leq n} (1-\mathbf{t}^{\deg H_{ij}})} =
\frac{1}{\prod\limits_{i=1}^n (1-t_i)^{n-2}
\prod\limits_{1\le i<j\le n} (1-t_i t_j)}.
\end{equation}
In particular, $h_{Q_n}({\bf s})$ is  the coefficient of
$t_1^{s_1}\cdots t_n^{s_n}$ in its Taylor expansion.
\end{lemma}
\begin{remark}\label{rem: pass to a smaller ring}
The factors $(1-t_i)^{-(n-2)}$ in the Hilbert series ${\rm HS}_{Q_n}({\bf t}) = \sum  h_{Q_n}({\bf s})\, {\bf t}^{\bf s}$ \eqref{eq: HS conformal} can be expanded using the following binomial series
$$
\frac{1}{(1-t)^{n-2}}=\sum_{k\ge0}\binom{n-2+k-1}{k}t^k.
$$
As a result, their contribution to the coefficient of $t_1^{s_1}\cdots t_n^{s_n}$ is entirely
encoded by explicit binomial coefficients. Indeed, when extracting the coefficient of the monomial $t_1^{s_1}\cdots t_n^{s_n}$, the contribution of the
variables $\mathcal{V}_{ij}$ amounts to summing over all decompositions
$s_i=b_i+k_i$ with $k_i\ge0$, weighted by the binomial coefficient
$\binom{n-3+k_i}{k_i}$. Hence, in order to obtain closed‑form expressions,
it suffices to analyse the generating function associated with the variables $H_{ij}$ alone,
namely the Hilbert series of the multigraded ring $R_n=\K[H_{ij}]$:
\begin{equation}
    {\rm HS}_{R_n}({\bf t}) = \sum  h_{R_n}({\bf b})\, {\bf t}^{\bf b} = \frac{1}{\prod\limits_{1\leq i<j\leq n} (1-\mathbf{t}^{\deg H_{ij}})} = \frac{1}{\prod\limits_{1\le i<j\le n} (1-t_i t_j)}.
\end{equation}
\end{remark}
The $\mathbb{Z}^n$‑grading on the ring
$
R_n=\K[H_{ij}]
$
is induced from the grading on $Q_n$ given in~\eqref{eq: multidegrees VH} by restriction.
It is therefore determined by an $n\times\binom{n}{2}$ submatrix $A$ of the matrix $B$~\eqref{eq:matB}, whose columns are indexed by the variables $H_{ij}$ and are given by their degrees:
\begin{equation}\label{eq:matA}
A=(a_{k,*}), \quad \text{where } \;
a_{k,H_{ij}} =
\begin{cases}
1, & k=i \text{ or } k=j,\\
0, & \text{otherwise}.
\end{cases}
\end{equation}
We conclude this subsection with a motivating example.
\begin{example}[$n=3$] \label{ex: conformal polytope n=3}
Consider the polynomial ring
$
Q_3 = \K[\mathcal{V}_{13},\mathcal{V}_{21},\mathcal{V}_{32},H_{12},H_{23},H_{13}],
$
equipped with the $\mathbb{Z}^3$‑grading defined by
\begin{align*}
\deg(\mathcal{V}_{13})=(1,0,0), &\qquad \deg(H_{12})=(1,1,0),\\
\deg(\mathcal{V}_{21})=(0,1,0), &\qquad \deg(H_{23})=(0,1,1),\\
\deg(\mathcal{V}_{32})=(0,0,1), &\qquad \deg(H_{13})=(1,0,1).
\end{align*}
For ${\bf s}=(s_1,s_2,s_3)$, the space $(Q_3)_{\bf s}$ of homogeneous polynomials of multidegree
${\bf s}$ is finite\-dimensional. A basis is given by monomials
$
\mathcal{V}_{13}^{y_1} \mathcal{V}_{21}^{y_2} \mathcal{V}_{32}^{y_3} H_{12}^{x_{12}} H_{23}^{x_{23}} H_{13}^{x_{13}}
$
whose exponents satisfy
$$
B\,({\bf y}\mid{\bf x}) = {\bf s}, \quad \text{that is, } \;
\begin{cases}
y_1+x_{12}+x_{13}=s_1,\\
y_2+x_{12}+x_{23}=s_2,\\ 
y_3+x_{23}+x_{13}=s_3,  
\end{cases} 
\quad
\text{for } \; y_1,y_2,y_3,x_{12}, x_{23}, x_{13}\ge0.  
$$
Eliminating the variables $y_1,y_2,y_3$ shows that counting such monomials is equivalent to counting
integer solutions $(x_{12}, x_{23}, x_{13})\in\mathbb{Z}_{\ge0}^3$ of the following system of inequalities:
\begin{equation}\label{eq: FM3}
x_{12}+ x_{13}\le s_1,\qquad x_{12}+x_{23}\le s_2,\qquad x_{23}+ x_{13}\le s_3.
\end{equation}
The Hilbert function 
$
h_{Q_3}({\bf s})=\dim_\C (Q_3)_{\bf s}
$
counts non-negative integer solutions of~\eqref{eq: FM3}, which are also the lattice points in the convex polytope
$
\mathcal{C}_3(s_1,s_2,s_3)$ defined by~\eqref{eq: FM3} and $x_{ij}\geq 0$.
\end{example} 

\subsection{A geometric perspective: lattice points in polytopes}\label{per: geo}
Example~\ref{ex: conformal polytope n=3} suggests a geometric reformulation of the counting problem above.
For general~$n$, the Hilbert function $h_{Q_n}({\bf s})$ can be interpreted as the number of
lattice points in a family of convex polytopes depending on integer non-negative parameters 
${\bf s}=(s_1,\ldots,s_n)\in\mathbb{Z}_{\ge0}^n$.
\begin{definition} \label{def: conformal polytope}
For ${\bf s}\in\mathbb{Z}_{\ge0}^n$, we define the \emph{$n$‑point conformal polytope}
$\mathcal{C}_n({\bf s})$ to be the set 
\begin{equation} \label{eq: conformal polytope}
    \mathcal{C}_n({\bf s}) \coloneqq \{({\bf \hat{y}}, {\bf x}) \in \R_{\ge 0}^{n(n-3) + {n \choose 2}}: \; \hat{B}\,({\bf \hat{y}}\mid{\bf x})\le {\bf s}\},
\end{equation}
where $\hat{B}$ is the submatrix of $B$ defined in~\eqref{eq:matB} obtained by deleting $n$ columns corresponding to some $\mathcal{V}_{1j_1}, \ldots, \mathcal{V}_{nj_n}$. 
Equivalently,
$\mathcal{C}_n({\bf s})$ is the projection of the solution set to $B\,({\bf y}\mid {\bf x}) ={\bf s}$, where $({\bf y}, {\bf x})$ is a point in $ \R_{\ge 0}^{n(n-2) + {n \choose 2}}$, 
onto the $n(n-3) + {n \choose 2}$ many coordinates $({\bf \hat{y}}, {\bf x})$.
\end{definition}
Polytopes of this type have appeared previously in the combinatorial literature,
where they are known as \emph{fractional ${\bf s}$‑matching polytopes}. Let $G=(V,E)$ be a finite undirected graph. The ordinary \emph{matching polytope} associated with a graph $G$ is the convex hull of incidence vectors of \emph{matchings}, that is, of subsets of pairwise vertex-disjoint edges. Equivalently, a matching is encoded by a vector $x\in\{0,1\}^{|E|}$ satisfying the following vertex constraints:
$$
\sum_{e\in E_v}x_e\leq 1
\qquad\text{for all }v\in V.
$$ 
Conformal polytopes $\mathcal{C}_n({\bf s })$ introduced above have a similar form. We now elaborate on this. 

\medskip

A \emph{fractional matching} is encoded by a vector
$x\in\mathbb{R}_{\ge0}^{|E|}$ satisfying the following inequalities 
$$
\sum_{e\in E_v} x_e \le 1 \qquad \forall v\in V,
$$
where $E_v$ denotes the set of edges incident to $v$.
More generally, given a vertex weight function ${\bf s}:V\to\mathbb{Z}_{\ge0}$, a
\emph{fractional ${\bf s}$‑matching} is a vector $x\in\mathbb{R}_{\ge0}^{|E|}$ such that
$$
\sum_{e\in E_v} x_e \le {\bf s}(v) \qquad \forall v\in V.
$$
\begin{definition}
A \emph{fractional ${\bf s}$-matching polytope} ${\rm FM}_{\bf s}(G)$ is a convex polytope given by the convex hull of the  fractional ${\bf s}$‑matchings of the graph $G$, that is, 
\begin{equation}\label{eq: fractional s-matching}
{\rm FM}_{\bf s}(G) \coloneqq
\left\{
x \in \mathbb{R}_{\ge 0}^{|E|} :
\forall v \in V \; \sum_{e \in E_v} x_e \le {\bf s}(v)
\right\}.    
\end{equation}       
\end{definition}
These polytopes are bounded rational polytopes with half-integer vertices, see, for example,~\cite{Balinski1965} for details.
A \emph{fractional perfect ${\bf s}$‑matching polytope} is obtained from a fractional ${\bf s}$‑matching polytope by replacing all  inequalities by equalities, we use the notation from \cite{BEHREND20133822}:
\begin{equation}\label{eq: fractional perfect s-matching}
 \mathcal{P}(G,{\bf s}) \coloneqq \left\{
x \in \mathbb{R}_{\ge 0}^{|E|} :
\forall v \in V \; \sum_{e \in E_v} x_e = {\bf s}(v)
\right\}.   
\end{equation}
\begin{remark}
The three‑point conformal polytope $\mathcal{C}_3(s_1,s_2,s_3)$~\eqref{eq: FM3} coincides with the fractional ${\bf s}$‑matching
polytope of the complete graph $K_3$.
More generally, the conformal polytope $\mathcal{C}_n({\bf s})$ can be interpreted as the fractional
${\bf \hat{s}}$‑matching polytope of the augmented graph $K_n^*$.
The graph $K_n^*$ is obtained from the complete graph $K_n$ by attaching $n-3$ pending edges to
each vertex. It has $n+n(n-3)$ vertices and $n(n-3)+\binom{n}{2}$ edges.
The vertex weight function ${\bf \hat{s}}$ assigns weight $s_i$ to the original vertex $v_i$ of
$K_n$ and a sufficiently large integer to each of the new degree‑one vertices, so that the
corresponding inequalities in the hyperplane description of the polytope
${\rm FM}_{\bf \hat{s}}(K_n^*)$ become redundant. See Example~\ref{ex: Cs1s2s3s4} for details.
\end{remark}

\medskip

We now illustrate the above discussion in the simplest non‑trivial case for $n=3$.
The following closed‑form formula was stated in~\cite[eq.~4.20]{CPPR}. We give a proof for completeness.
\begin{proposition}\label{prop: Ns1s2s3}
Let ${\bf s}=(s_1,s_2,s_3)$ with $s_1\le s_2\le s_3$, and set
$
p=\max(0,\,s_1+s_2-s_3).
$
Then the Hilbert function value
$
h_{Q_3}({\bf s})
= \bigl|{\cal C}_3({\bf s})\cap\mathbb{Z}^3\bigr| = 
\bigl|{\rm FM}_{\bf s}(K_3)\cap\mathbb{Z}^3\bigr|
$
is given by
\begin{equation}\label{eq: N b1b2b3}
N(s_1,s_2,s_3)
=
\frac{(s_1+1)(s_1+2)(3s_2-s_1+3)}{6}
-
\frac{p(p+2)(2p+5)}{24}
-
\frac{1-(-1)^p}{16}.
\end{equation}
\end{proposition}
\begin{proof}
We count lattice points in the convex polytope defined by
$$
\bigl\{(x,y,z)\in\mathbb{Z}_{\ge0}^3
\mid y+z\le s_1,\; x+z\le s_2,\; x+y\le s_3\bigr\}.
$$
Fix $z$. Ignoring the constraint $x+y\le s_3$, the remaining inequalities define a rectangle
$[0,s_2-z]\times[0,s_1-z]$ with $(s_2-z+1)(s_1-z+1)$ lattice points.
Summing over $z$ yields
$$
N_0=\sum_{z=0}^{s_1}(s_2-z+1)(s_1-z+1)
=\frac{(s_1+1)(s_1+2)(3s_2-s_1+3)}{6}.
$$
If $s_1+s_2\le s_3$, the inequality $x+y\le s_3$ is automatically satisfied and $N_0$ is the
desired count. Otherwise, the line $x+y=s_3$ cuts off a triangular region from each rectangle.
Setting $p=\max(0,s_1+s_2-s_3)$, the number of lattice points to be subtracted is
$$
\sum_{z=0}^{\lfloor (p-1)/2\rfloor}\frac{(p-2z)(p-2z+1)}{2}
=
\frac{p(p+2)(2p+5)}{24}
+
\frac{1-(-1)^p}{16}.
$$
Subtracting this correction from $N_0$ gives~\eqref{eq: N b1b2b3} and completes the proof.
\end{proof}
\noindent
We now move on to $R_n=\K[H_{ij}]$, since by Remark~\ref{rem: pass to a smaller ring} it suffices to study its Hilbert function.

\begin{proposition}\label{prop: HF perfect matching}
For ${\bf b} \in\mathbb{Z}_{\ge0}^n$, the Hilbert function of the ring
$
R_n=\K[H_{ij}]
$
satisfies
$$
h_{R_n}({\bf b})
=
\bigl|\mathcal{P}(K_n,{\bf b})\cap\mathbb{Z}^{\binom{n}{2}}\bigr|,
$$
where $\mathcal{P}(K_n,{\bf b})$ denotes the fractional perfect ${\bf b}$‑matching polytope of the
complete graph $K_n$.
\end{proposition}

\begin{proof}
By the multigrading \eqref{eq: multidegrees VH}, a monomial
$\prod_{i<j}H_{ij}^{x_{ij}}$ has multidegree ${\bf b}$ exactly when
$$ \textstyle
\sum_{j\neq i} x_{ij} = b_i \quad \text{for all } \;i=1,\ldots,n.
$$
Thus degree-${\bf b}$ monomials correspond to the lattice points of
${
\mathcal{P}(K_n,{\bf b})
=
\{x\in\mathbb{R}_{\ge0}^{\binom{n}{2}}:Ax={\bf b}\}},
$
with $A$ the incidence matrix of $K_n$ from \eqref{eq:matA}. Counting these
monomials gives $h_{R_n}({\bf b})$.
\end{proof}
\begin{example}\label{ex: polytopes equal spins}
Consider the uniform case
$
{\bf s}=(s,s,\ldots,s).
$
Then, the fractional (perfect) ${\bf s}$‑matching polytopes
${\rm FM}_{\bf s}(G)$ and $\mathcal{P}(G,{\bf s})$ are  the  $(s/2)$‑dilations of the lattice polytopes 
\[
{\rm FM}_{\bf 2}(G) =
\left\{
x \in \mathbb{R}_{\ge 0}^E :
\forall v \in V \; \sum_{e \in E_v} x_e \le 2 
\right\}, \quad
\mathcal{P}(G, {\bf 2}) =
\left\{
x \in \mathbb{R}_{\ge 0}^E :
\forall v \in V \; \sum_{e \in E_v} x_e = 2
\right\},
\]
where $
{\bf 2}=(2,2,\ldots,2).
$
When $s$ is even, the number of lattice points in
${\rm FM}_{\bf s}(G)$ and $\mathcal{P}(G,{\bf s})$ is given exactly by evaluating the
corresponding Ehrhart polynomials~\eqref{eq: Ehrhart preliminaries}
$$L_{\mathrm{np}}(t) \coloneqq L({\rm FM}_{\bf 2}(G), t) \quad   \text{and} \quad L_{\mathrm{p}}(t) \coloneqq L(\mathcal{P}(G,{\bf 2}),t)$$
at
$t=s/2$.
For $G=K_n$ with $n=3,\ldots,7$, we computed these Ehrhart polynomials using the software \texttt{LattE}~\cite{Latte}. The results are listed below, with full polynomials available at \cite{zenodo}.
\begin{footnotesize}
\begin{align*}
n=3:\quad
 & L_{\mathrm{p}}(t) \;\,=1, \\
 & L_{\mathrm{np}}(t)=\tfrac{1}{2}(t+1)(4t^2+5t+2), \\[5pt]
n=4:\quad
 & L_{\mathrm{p}}(t)\;\,=(t+1)(2t+1), \\
 & L_{\mathrm{np}}(t)=\tfrac{1}{9}(t+1)^2(8t^4+32t^3+47t^2+30t+9), \\[5pt]
n=5:\quad
& L_{\mathrm{p}}(t)\;\,=\tfrac{1}{24}(t+1)(15t^4+60t^3+95t^2+70t+24), \\
 & L_{\mathrm{np}}(t)=\tfrac{1}{24192}(t+1)(t+2)
(2920t^8+\ldots+66600t+12096),  \\[5pt]
n=6:\quad
& L_{\mathrm{p}}(t) \;\,=\tfrac{1}{7560}(t+1)(t+2)
(608t^7+5016t^6+\ldots+17469t+3780), \\
& L_{\mathrm{np}}(t)=\tfrac{1}{3891888000}(t+1)(t+2)(2t+3)
(7796896t^{12}+140344128t^{11}+\ldots+648648000), \\[5pt]
n=7:\quad
& L_{\mathrm{p}}(t) \;\,=\tfrac{1}{19958400}(t+1)(t+2)
(47970t^{12}+863460t^{11}+\ldots+9979200),\\
& L_{\mathrm{np}}(t)=\tfrac{1}{21722339358720000}{(t+1)(t+2)(t+3)}
(586746749696t^{18}+\ldots+3620389893120000).
\end{align*}
\end{footnotesize}
Recall that $L(P,-t)$ counts the interior lattice points of the $t$-dilate of
$P$. Thus the above computations show that most of these polytopes have no
interior lattice points.
\end{example}
Unfortunately, the elementary geometric reasoning used in Proposition~\ref{prop: Ns1s2s3} does not na\-turally generalize to the case $n>3$. To address this, we will need more advanced machinery.
Namely, we reinterpret lattice-point counting problems in terms of vector partition functions.
Let $A\in\mathbb{Z}^{n\times d}$ be an integer matrix.  
The associated \emph{vector partition function} is defined by
\begin{equation}\label{eq: defVectorPart}
\varphi_A({\bf b})
\coloneqq\bigl|\{{\bf x}\in\mathbb{Z}_{\ge0}^d : A{\bf x}={\bf b}\}\bigr|,
\quad {\bf b}\in\mathbb{Z}^n.    
\end{equation}
The following theorem describes the structure of such functions.
\begin{theorem}[\cite{STURMFELS1995302}]
    Let $A$ be an $n \times d$ integer matrix of rank $n$. The vector partition function of~$A$, denoted $\varphi_A$, is a piecewise quasi‑polynomial of degree $d-n$ whose domains of quasi‑polynomiality are the maximal cones (chambers) in the chamber complex of~$A$.
\end{theorem}
In practice, these quasi‑polynomials and their chambers can be computed using the software
package \texttt{barvinok}~\cite{Verdoolaege2007}. We illustrate this in the following examples.
\begin{example}
Consider the incidence matrix $A$~\eqref{eq:matA} of a complete graph $K_n$. Then the vector partition function $\varphi_{A}({\bf b})$ counts the number of lattice points in the polytope $\mathcal{P}(K_n,{\bf b})$. 

\smallskip

In Table~\ref{tab: K4 chambers} we list the quasi‑polynomial pieces of $\varphi_A({b_1,b_2,b_3,b_4})$ and the
corresponding chamber decomposition for $K_4$.
There are eight chambers in total. After imposing the ordering $b_1\le b_2\le b_3\le b_4$, the number of distinct chambers
reduces to two.
\begin{table}[ht]
\centering
\renewcommand{\arraystretch}{1.4}
\small
\begin{tabular}{c p{1.5cm} p{12cm}}
\hline
\textbf{Ch.} & $\varphi_A({\bf b})$ & \textbf{Chamber constraints} \\
\hline
1 &
${b_1 + 2 \choose 2}$ &
$
 b_4 \ge b_1 + b_2 - b_3 \text{ and } b_1 + b_3 -b_2 \le b_4 <  b_2 + b_3-b_1
$ \\[2pt]

2 &
${b_2 + 2 \choose 2}$ &
$
 b_4 \ge b_1 + b_2 - b_3 \text{ and } b_2 + b_3 -b_1 \le b_4 <  b_1 + b_3-b_2
$ \\[2pt]

3 &
${b_3 + 2 \choose 2}$ &
$
 b_4 \ge b_1 + b_3 - b_2 \text{ and } b_2 + b_3 -b_1 \le b_4 <  b_1 + b_2-b_3
$ \\[2pt]

4 &
${b_4 + 2 \choose 2}$ &
$
 b_4 < \min(b_2 + b_3 - b_1, \; b_1 + b_3 -b_2, \;b_1 + b_2-b_3)
$ \\[2pt]

5 &
$\tfrac{1}{8}Q_1({\bf b})$ &
$b_4 \ge b_1 - b_2 - b_3 \text{ and } b_2 + b_3 - b_1 \le b_4 < \min(b_1+b_3 - b_2, \;b_1 + b_2 - b_3)$ \\[2pt]

6 &
$\tfrac{1}{8}Q_2({\bf b})$ &
$b_4 \ge b_1 + b_3 - b_2 \text{ and } b_2 - b_1 - b_3 \le b_4 < \min(b_2+b_3 - b_1, \; b_1 + b_2 - b_3)$ \\[2pt]

7 &
$\tfrac{1}{8}Q_3({\bf b})$ &
$b_4 \ge b_1 + b_2 - b_3 \text{ and } b_3 - b_1 - b_2 \le b_4 < \min(b_2+b_3 - b_1, \; b_1 + b_3 - b_2)$ \\[2pt]

8 &
$\tfrac{1}{8}Q_4({\bf b})$ &
$b_4 \ge b_1 + b_2 - b_3 \text{ and } b_1 + b_3 - b_2 \le b_4 \le b_1+b_2+b_3$ \\[1pt]
\hline
\end{tabular}
\caption{The vector partition function $\varphi_A({\bf b})$ for the incidence matrix of $K_4$.
The quadratic polynomials are given by
$
Q_i({\bf b})=(|{\bf b}|-2b_i+2)(|{\bf b}|-2b_i+4)$ with
$|{\bf b}|=b_1+b_2+b_3+b_4$.
In all cases we impose ${\bf b}\in\mathbb{Z}_{\ge0}^4$ and $|{\bf b}|\equiv0\pmod2$, since for odd $|{\bf b}|$ one has $\varphi_A({\bf b}) = 0$.}
\label{tab: K4 chambers}
\end{table}

We carried out analogous computations for $n=3,4,5$, see~\cite{zenodo}.
The corresponding vector partition functions have $1,8,345$ chambers, respectively.
After imposing $b_1\le\cdots\le b_n$, the number of chambers reduces to
$1,2,6$, respectively.
\end{example}
\begin{example}\label{ex: Cs1s2s3s4}
For $n=4$ the conformal polytope $\mathcal{C}_4({\bf s})$ is a fractional
${\bf \hat{s}}$‑matching polytope of an augmented graph $K_4^*$ obtained from the complete graph $K_4$ by attaching one pending edge
to each vertex. The extended vertex weight function
$
{\bf \hat{s}}=(s_1,s_2,s_3,s_4,s,s,s,s)
$
is chosen so that $s\gg s_i$. The corresponding graph is illustrated below.

\smallskip

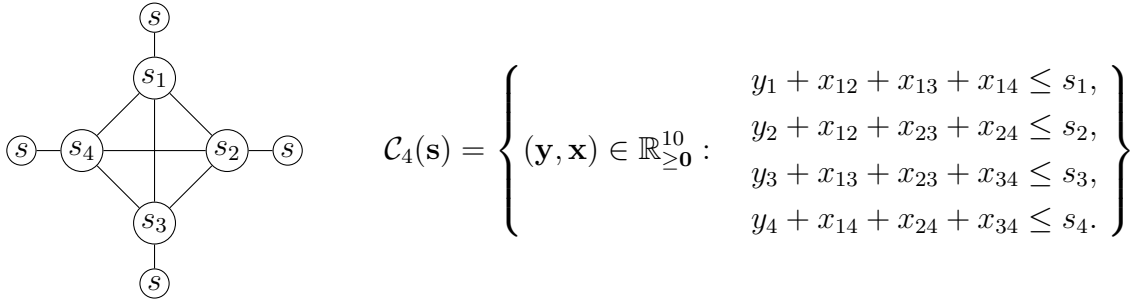
\begin{minipage}[]{0.35\textwidth}
\begin{center}
\begin{tikzpicture}[scale=0.8, every node/.style={circle, draw, inner sep=1.2pt}]
\node (v1) at (0,1.2) {$s_1$};
\node (v2) at (1.2,0) {$s_2$};
\node (v3) at (0,-1.2) {$s_3$};
\node (v4) at (-1.2,0) {$s_4$};

\node (p1) at (0,2.2) {$s$};
\node (p2) at (2.2,0) {$s$};
\node (p3) at (0,-2.2) {$s$};
\node (p4) at (-2.2,0) {$s$};

\draw (v1)--(v2)--(v3)--(v4)--(v1);
\draw (v1)--(v3);
\draw (v2)--(v4);

\draw (v1)--(p1);
\draw (v2)--(p2);
\draw (v3)--(p3);
\draw (v4)--(p4);
\end{tikzpicture}
\end{center}
\end{minipage}
\begin{minipage}[b]{0.5\textwidth}
$$\mathcal{C}_4({\bf s}) = \left\{({\bf y},{\bf x}) \in \R_{\ge {\bf 0}}^{10} : \quad
\begin{aligned}
y_1+x_{12}+x_{13}+x_{14} &\le s_1,\\
y_2+x_{12}+x_{23}+x_{24} &\le s_2,\\
y_3+x_{13}+x_{23}+x_{34} &\le s_3,\\
y_4+x_{14}+x_{24}+x_{34} &\le s_4.
\end{aligned}\; \right\}
$$
\end{minipage}

\vspace{0.2cm}
\noindent
Assuming $s_1\le s_2\le s_3\le s_4$, the corresponding vector partition function has $14$ chambers. Explicit expressions are available at the link~\cite{zenodo}.
If we further specialise to the case
$
{s_2=s_3=s_4}
$,
the chamber structure collapses to just three regions,
 listed in
Table~\ref{tab: C4 chambers}.
\begin{table}[h]
\centering
\renewcommand{\arraystretch}{1.4}
\small
\begin{tabular}{c p{10cm} p{4cm}}
\hline
\textbf{Ch.} & $\varphi_B({\bf s})$ & \textbf{Chamber constraints} \\
\hline
1 &
$1$ &
$
s_1=s_2=0
$ \\

2 &
$\tfrac{1}{20}{s_2 + 3 \choose 3}(s_2+1)(11s_2^2+44s_2+40)$ &
$
 s_1=0 \text{ and } s_2 > 0
$ \\

3 &
$\tfrac{s_2}{960}(22s_2^5+220s_2^4+855s_2^3+1620s_2^2+1508s_2+575) - 1/32 \lfloor \tfrac{s_2}{2} \rfloor $ &
$
 s_1=1 \text{ and } s_2>0
$ \\[1pt]
\hline
\end{tabular}
\caption{Number of lattice points in the conformal polytope 
$
\mathcal{C}_4(s_1,s_2,s_2,s_2)$ with $s_1\le s_2$.
}
\label{tab: C4 chambers}
\end{table}
Finally, in the uniform case $s_1=s_2=s_3=s_4$, the number of lattice points in $\mathcal{C}_4({\bf s})$ is given by
\begin{align*}
\varphi_B({\bf s}) &= \tfrac{1}{10!}\, (s+2)(s+4)\bigl(
        878\,s^{8}+16682\,s^{7}+135239\,s^{6}+609710\,s^{5} \\
     &\hspace{8.4em} +\,1669742\,s^{4}+2843828\,s^{3}+2952216\,s^{2} \\
     &\hspace{8.4em} +\,1727280\,s+453600\bigr),
        \qquad  \qquad \qquad \qquad \qquad \qquad \text{if }s\text{ is even}, \\[4pt]
\varphi_B({\bf s}) &= \tfrac{1}{10!}\,(s+1)(s+3)\bigl(
        878\,s^{8}+18438\,s^{7}+165969\,s^{6}+835410\,s^{5} \\
     &\hspace{8.4em} +\,2570367\,s^{4}+4952262\,s^{3}+5852971\,s^{2} \\
     &\hspace{8.4em} +\,3922530\,s+1185975\bigr),
        \qquad  \qquad \qquad \qquad \qquad \qquad \text{if }s\text{ is odd}.
\end{align*}
\end{example}

\subsection{A combinatorial perspective: Kostka numbers}\label{per: com}
We now move on to combinatorics. We start by introducing necessary combinatorial objects.
A \emph{partition} $\lambda=(\lambda_1,\lambda_2,\ldots)$ is a weakly decreasing sequence of non‑negative integers with finite sum $|\lambda|=\sum_i \lambda_i$.
The \emph{Young tableau} of $\lambda$ is a left‑justified array of boxes with $\lambda_i$ boxes
in the $i$th row. A \emph{semistandard Young tableau} of shape $\lambda$ is a filling of the boxes of the Young
tableau of $\lambda$ with positive integers such that the entries are weakly increasing along
each row and strictly increasing down each column.
The \emph{weight} of a tableau is the vector recording the number of occurrences of each integer.
The \emph{conjugate partition} $\lambda'$ is obtained by transposing the Young tableau.
We say that a partition $\lambda$ is \emph{even} if all entries $\lambda_i$ are even.
\begin{figure}[ht]
    \centering
\begin{tikzpicture}[scale=0.3]
\foreach \x/\y in {0/0,1/0,2/0,3/0, 0/-1,1/-1, 0/-2,1/-2}
  \draw (\x,\y) rectangle ++(1,1);
\node at (-4.5,-0.5) {$\lambda=(4,2,2)$};
\end{tikzpicture}
\qquad
\begin{tikzpicture}[scale=0.3]
\foreach \x/\y in {0/0,1/0,2/0, 0/-1,1/-1,2/-1, 0/-2, 0/-3}
  \draw (\x,\y) rectangle ++(1,1);
\node at (8.2,-1.5) {$\lambda'=(3,3,1,1)$};
\end{tikzpicture}
    \caption{$\lambda$ is even, while its conjugate $\lambda'$ is not. }
    \label{fig:placeholder}
\end{figure}

\noindent
\emph{Schur polynomials} are certain symmetric polynomials indexed by partitions.
For a partition~$\lambda$, the Schur polynomial $s_\lambda$ can be written as a sum
of monomial symmetric functions:
\begin{equation}\label{eq: schur sum of monomial func}
 s_\lambda(x_1,\ldots,x_n) = \sum\limits_{\mu} K_{\lambda\mu}\, m_\mu (x_1,\ldots,x_n),  \quad m_\mu = \sum_{\sigma \in S_n} x_1^{\sigma(\mu_1)} \cdots x_n^{\sigma(\mu_n)},  
\end{equation}
where $K_{\lambda\mu}$ are the \emph{Kostka numbers} counting  semistandard Young tableaux of shape $\lambda$
and weight $\mu$.
The relevance of these notions to our setting is explained by the following result.
\begin{lemma}\label{lem: Hilbert-Kostka}
For ${\bf b}\in\mathbb{Z}_{\ge0}^n$, the Hilbert function of the ring $R_n=\K[H_{ij}]$ satisfies
$$
h_{R_n}({\bf b})
=
\sum_{\lambda\in\mathcal{P}'_{\mathrm{even}}}
K_{\lambda,{\bf b}},
$$
where $\mathcal{P}'_{\mathrm{even}}$ denotes the set of partitions $\lambda$ of $n$ whose conjugate
partition $\lambda'$ is even.
\end{lemma}
\begin{proof}
    By \cite[Lemma 8.16]{CCA}, the Hilbert function $h_{R_n}({\bf b})$ is the coefficient of $t_1^{b_1}t_2^{b_2}\cdots t_n^{b_n}$ in 
    \[
     {\rm HS}_{R_n}({\bf t}) = \sum  h_{R_n}({\bf b})\, {\bf t}^{\bf b} = \frac{1}{\prod\limits_{1\leq i<j\leq n} (1-\mathbf{t}^{\deg H_{ij}})} = \frac{1}{\prod\limits_{1\le i<j\le n} (1-t_i t_j)}.
    \]
A variant of Cauchy’s identity (see~\cite[p.~77]{Macdonald}) gives the expansion
$$
\frac{1}{\prod_{1\le i<j\le n} (1-t_i t_j)}
=
\sum_{\lambda\in\mathcal{P}'_{\mathrm{even}}} s_\lambda.
$$
Since each Schur polynomial $s_\lambda$ is homogeneous of total degree $|\lambda|$, the
homogeneous component of total degree $|{\bf b}|$ in this series is the sum of $s_\lambda$ over
partitions $\lambda$ of $|{\bf b}|$ whose conjugate $\lambda'$ is even.
Expanding each $s_\lambda$ in the symmetric monomial functions basis as
$
s_\lambda=\sum_{\mu} K_{\lambda\mu}\, m_\mu,
$
and extracting the coefficient of the monomial $t_1^{b_1}\cdots t_n^{b_n}$ yields the stated formula.
\end{proof}
\begin{corollary}
The number of lattice points in the fractional perfect ${\bf b}$-matching polytope~is
\[
|\mathcal{P}(K_n, {\bf b}) \cap \Z^{n \choose 2}| = \varphi_A({\bf b}) = \sum_{\lambda\in\mathcal{P}'_{\mathrm{even}}}
K_{\lambda,{\bf b}},
\]
where $\lambda$ is a partition of $|{\bf b}|$ with even conjugate and $\varphi_A({\bf b})$ is a vector partition function~\eqref{eq: defVectorPart}.
\end{corollary}
\begin{example}
The first graded pieces of the expansion of the Hilbert series of
$
R_4
$ are
\begin{center}
\vspace{-0.5cm}
\begin{minipage}[]{0.35\textwidth}
\begin{align*}
\deg 0:\quad & s_{\varnothing},\\[2pt]
\deg 2:\quad & s_{(1,1)},\\[2pt]
\deg 4:\quad & s_{(2,2)} + s_{(1,1,1,1)},\\[2pt]
\deg 6:\quad & s_{(3,3)},
\end{align*}      
\end{minipage} 
\begin{minipage}[]{0.5\textwidth}
\begin{align*}
\deg 8:\quad &
s_{(4,4)} + s_{(3,3,1,1)} + s_{(2,2,2,2)},\\[2pt]
\deg 10:\quad &
s_{(5,5)} + s_{(4,4,1,1)} + s_{(3,3,2,2)},\\[2pt]
\deg 12:\quad &
s_{(6,6)} + s_{(5,5,1,1)} + s_{(4,4,2,2)} + s_{(3,3,3,3)},\\[2pt]
\deg 14:\quad &
s_{(7,7)} + s_{(6,6,1,1)} + s_{(5,5,2,2)} + s_{(4,4,3,3)}.
\end{align*}    
\end{minipage}
\end{center}
For example, the degree 4 piece is $(m_{(2,2)} + m_{(2,1,1)} + 2\,m_{(1,1,1,1)}) + m_{(1,1,1,1)}$. Thus,
\[
    h_{R_n}(2,2,0,0) = 1, \quad h_{R_n}(2,1,1,0) = 1, \quad h_{R_n}(1,1,1,1) = 3.
\]
\end{example}
For ${\bf b}={\bf 1}=(1,\ldots,1)$, the Kostka numbers $K_{\lambda,{\bf 1}}$ count the number of
\emph{standard Young tableaux} of shape $\lambda$.
This number is given by the well-known hook‑length formula~\cite{FrameRobinsonThrall1954}:
$$
K_{\lambda,{\bf 1}}
=
\frac{|\lambda|!}{\prod_{(i,j)\in\lambda} h_\lambda(i,j)},
$$
where $h_\lambda(i,j)$ denotes the \emph{hook length} of the cell $(i,j)$ in the Young tableau of
$\lambda$, that is, the number of cells directly to the right of $(i,j)$ and  below it,
including the cell $(i,j)$ itself.
We therefore obtain the following corollary.
\begin{corollary}
The number of lattice points in the fractional perfect matching polytope of the complete graph
$K_n$ satisfies
$
\bigl|\mathcal{P}(K_n)\cap\mathbb{Z}^{\binom{n}{2}}\bigr|
=
h_{R_n}({\bf 1})
=
\sum_{\lambda\in\mathcal{P}'_{\mathrm{even}}}
\frac{n!}{\prod_{(i,j)\in\lambda} h_\lambda(i,j)}.
$
\end{corollary}
\begin{remark}\label{rem: transportation polytope}
 Another interpretation of the value $ h_{R_n}(\mathbf{b}) $ is as the number of symmetric ${n\times n}$ matrices with zero diagonal and non-negative integer entries having prescribed row sums~$ b_i $. That is, the fractional perfect 
$\bf{b}$-matching polytope $\mathcal{P}(K_n, \bf{b})$  is a symmetric analogue of the classical \emph{transportation polytope}~\cite{BOLKER1972251}. These polytopes naturally appear when counting conformally covariant tensor structures in $d=4$, see \cite[Equation 15]{Heslop2026}. When $\mathbf{b} = (1,\dots,1)$, the polytope $\mathcal{P}(K_n, \bf{b})$ is the symmetric analogue of the well-known \emph{Birkhoff polytope}. In the uniform case $ \mathbf{b} = (b,\ldots,b) $, \cite[Theorem~1.1]{McKay2011ASYMPTOTICEO} gives an asymptotic formula for the hilbert function $ h_{R_n}(\mathbf{b}) $ as $ n \to \infty $, assuming the row sum is sufficiently large. The authors further conjecture that their result holds for all row sums.
\end{remark}

\section{Building blocks as fundamental invariants} \label{sec:unipotent}
In the physics literature \cite{DanielNathan,CT,CPPR,DKKPS,Zhiboedov} one uses the fact that rational
$n$-point structures satisfying Lorentz invariance and transversality can be
expressed in terms of the \emph{basic building blocks}~\eqref{eq: confblocks intro}. We give an
algebraic proof of this statement using tools from invariant theory:  the building blocks generate
the field of rational invariants under
$Z_i\mapsto Z_i+\alpha_iP_i$.  
We use Weyl's theorems on Lorentz
invariants and Rosenlicht's theorem on rational invariants.

\medskip
Let $n\geq 2$ and $m=d+2\geq 5$, where 
$d$ denotes the number of spatial dimensions, and let $\K$ be an algebraically closed field
of characteristic zero. Consider configurations of $n$ embedding-space vectors~$P_i$ together with $n$ polarization vectors~$Z_i$:
$$
    (P_1,Z_1,\ldots,P_n,Z_n)\in W:=(\K^m)^{2n},
$$
where the Lorentz group $\mathrm{O}(1,m-1)$ acts on all vectors as follows:
$$P_i \mapsto \Lambda P_i \quad \text{ and } \quad Z_i \mapsto \Lambda Z_i \quad \text{ for } \; \Lambda \in \mathrm{O}(1, m{-}1).$$ 
We
write $(\,\cdot\,)$ for the Lorentzian scalar product. A classical theorem of  Weyl~\cite[p.~65]{Weyl} states that the ring of
Lorentz-invariant polynomials is generated by Lorentzian scalar products:
$$
    \K[P_i,Z_i]^{\mathrm{O}(1,m-1)}
    =
    \K[P_i\cdot P_j,\;Z_i\cdot Z_j,\;P_i\cdot Z_j].
$$
Moreover, these scalar products also generate the field of rational Lorentz
invariants:
\begin{equation*}
 \K(P_i,Z_i)^{\mathrm{O}(1,m-1)}
    =
    \mathrm{Frac}\,
    \bigl(\K[P_i\cdot P_j,\;Z_i\cdot Z_j,\;P_i\cdot Z_j]\bigr).
\end{equation*}
Indeed, the identity component of $\mathrm{O}(1,m-1)$, the \emph{restricted Lorentz group} $\mathrm{SO}^+(1,m-1)$, is  semisimple
 and hence has no non‑trivial characters, so the claim follows from \cite[Lemma~2.2]{CollSans}.

\medskip

The relations among these scalar products are called \emph{Gram relations}.
When
$2n\leq m$, the scalar products are algebraically independent, and in the case $2n>m$, the
ideal of Gram relations is generated by the $(m+1)\times(m+1)$ minors of the Gram
matrix \cite[p.~75]{Weyl},
\begin{equation}\label{eq: gram matrix}
\mathcal{G}
=
\begin{pmatrix}
P_i\cdot P_j & P_i\cdot Z_j\\
Z_i\cdot P_j & Z_i\cdot Z_j
\end{pmatrix}_{1\leq i,j\leq n}.
\end{equation}
The commutative algebra of these ideals was studied in~\cite{MSS}.
As described in Section~\ref{sec: preliminaries}, we~set
\begin{equation}\label{eq:null-cone}
    P_i^2=0,\quad P_i\cdot Z_i=0,\quad Z_i^2=0,
    \qquad i=1,\ldots,n.
\end{equation}
For the remaining scalar products, we introduce the following shorthand notation:
$$
    p_{ij}=P_i\cdot P_j=p_{ji},\quad
    q_{ij}=P_i\cdot Z_j,\quad
    r_{ij}=Z_i\cdot Z_j=r_{ji},
    \qquad i\neq j.
$$
Let
$
    S=\K[p_{ij},q_{ij},r_{ij}\mid i\neq j]
$ be the polynomial ring in these $n(2n - 2)$ variables,
and let $I_{\mathrm{Gram}}\subseteq S$ be the ideal obtained by evaluating the minors of~\eqref{eq: gram matrix} on the locus~\eqref{eq:null-cone}. We
write
\begin{equation}\label{eq:R-def}
    \mathcal{R}:=S/I_{\mathrm{Gram}},
    \quad
    \mathcal{F}:=\mathrm{Frac}(\mathcal{R}),
    \quad
    X:=\mathrm{Spec}(\mathcal{R}).
\end{equation}
Thus the affine variety $X$ parametrises all valid configurations of scalar products arising from null and transverse vectors in $\K^m$. 
By \cite[Theorem~2.5]{MSS}, for $m\geq 5$ we have
\begin{equation}\label{eq: trdegF}
\operatorname{trdeg}_{\K}\mathcal{F}
=
\dim X
=
\begin{cases}
n(2m-3)-\binom{m}{2}, & m<2n\\
n(2n-2), & m\geq 2n.
\end{cases}
\end{equation}
We now discuss the transversality constraint, that is, invariance under $Z_i \mapsto Z_i + \alpha_i P_i$.
\begin{definition}[Translation action]\label{def:transact}
The additive group
$(\K^n,+) \cong \mathbb{G}_a^n$
acts on $W$ by
\begin{equation}\label{eq:transact}
 \forall \,  \alpha = (\alpha_1,\ldots,\alpha_n) \in \mathbb{G}_a^n\colon
  \quad
 \alpha \cdot P_i = P_i,
  \quad
  \alpha \cdot Z_i = Z_i + \alpha_i\,P_i,
  \quad i=1,\ldots,n.
\end{equation}
The translation action \eqref{eq:transact} defines an injective
algebraic group homomorphism
\[
  \rho \colon \,  \mathbb{G}_a^n \;\rightarrow\; \mathrm{GL}_{2nm}(\K).
\]
Explicitly,
the image of $\alpha = (\alpha_1,\ldots,\alpha_n) \in \mathbb{G}_a^n$ is the
block-diagonal matrix
\begin{equation*}\label{eq:rho}
  \rho(\alpha)
  \;=\;
  \operatorname{diag}\left(
    \begin{pmatrix} {\rm I}_m & 0 \\ \alpha_1 {\rm I}_m & {\rm I}_m \end{pmatrix},\;
    \begin{pmatrix} {\rm I}_m & 0 \\ \alpha_2 {\rm I}_m & {\rm I}_m \end{pmatrix},\;
    \ldots,\;
    \begin{pmatrix} {\rm I}_m & 0 \\ \alpha_n {\rm I}_m & {\rm I}_m \end{pmatrix}
  \right)
  \;\in\;\mathrm{GL}_{2nm}(\K).
\end{equation*}
  In
particular, every element of $\rho(\mathbb{G}_a^n)$ is unipotent. Denote $\rho(\mathbb{G}_a^n)$ by $G$.
\end{definition}
\begin{proposition}\label{prop: freeaction} 
The action~\eqref{eq:transact} is generically free.
   In particular, it is faithful.
\end{proposition}
\begin{proof}
We show that the action of $G$ is free on the open dense subset
\[
  U = \{(P_1,Z_1,\ldots,P_n,Z_n)^{\top}\in W \mid \forall\,i : \, P_i \neq 0\} \subset W.
\]
Indeed, for arbitrary $x\in U$, as well as arbitrary $\alpha=(\alpha_1,\ldots,\alpha_n)\in\mathbb{K}^n$, we have 
    \begin{equation*}
        \rho(\alpha)x=(P_1,Z_1+\alpha_1P_1,\ldots,P_n,Z_n+\alpha_nP_n)^{\top}.
    \end{equation*}
    Since none of the $P_i$ are zero, the only $\alpha$ for which $\rho(\alpha)x=x$ is $\alpha=0$.
\end{proof}
Since the translation action~\eqref{eq:transact} commutes with the action of Lorentz group $\mathrm{O}(1, m{-}1)$, it induces corresponding transformation rules on the generators of the polynomial ring $S$.

\begin{proposition}\label{prop:transformation-rules}
For $\alpha \in \mathbb{G}_a^n$, the generators of $S$ transform as follows:
\begin{equation}\label{eq:transform}
\alpha \cdot p_{ij} = p_{ij}, \qquad
\alpha \cdot q_{ij} = q_{ij} + \alpha_j\, p_{ij}, \qquad
\alpha \cdot r_{ij} = r_{ij} + \alpha_i\, q_{ij} + \alpha_j\, q_{ji} + \alpha_i \alpha_j\, p_{ij}.
\end{equation}
\end{proposition}

\begin{proof}
The $p_{ij}$ are independent of the $Z_k$ and are therefore invariant. The remaining rules follow from bilinearity and symmetry of the Lorentzian scalar product:
\begin{align*}
P_i \cdot (Z_j + \alpha_j P_j) &= P_i \cdot Z_j + \alpha_j\, P_i \cdot P_j = q_{ij} + \alpha_j\, p_{ij}, \\
(Z_i + \alpha_i P_i) \cdot (Z_j + \alpha_j P_j) &= Z_i \cdot Z_j + \alpha_i\, P_i \cdot Z_j + \alpha_j\, Z_i \cdot P_j + \alpha_i \alpha_j\, P_i \cdot P_j \\
&= r_{ij} + \alpha_i\, q_{ij} + \alpha_j\, q_{ji} + \alpha_i \alpha_j\, p_{ij}. \qedhere
\end{align*}
\end{proof}

\begin{lemma}\label{lem: well-defined}
The action of $\mathbb{G}_a^n$ on $S$~\eqref{eq:transform} induces a well-defined action on the affine variety $X = \operatorname{Spec}(\mathcal{R})$ and on the field of fractions $\mathcal{F} = \operatorname{Frac}(\mathcal{R})$.
\end{lemma}

\begin{proof}
For each $\alpha \in \K^n$, the substitution rules \eqref{eq:transform} are polynomial in $\alpha$ and in the generators of $S$, and hence define a $\K$-algebra endomorphism $\phi_\alpha\colon S \to S$. One verifies directly from~\eqref{eq:transform} that $\phi_\alpha \circ \phi_\beta = \phi_{\alpha + \beta}$ and $\phi_0 = \mathrm{id}$, so each $\phi_\alpha$ is an automorphism with inverse $\phi_{-\alpha}$.

\smallskip

It remains to show that $I_{\mathrm{Gram}}$ is preserved by $\phi_\alpha$, that is, that $\phi_\alpha(I_{\mathrm{Gram}}) \subseteq I_{\mathrm{Gram}}$. To see this, observe that the translation action \eqref{eq:transact} is a linear automorphism of $W$, and so the Gram matrix $\mathcal{G}' \coloneqq \mathcal{G}(\alpha \cdot P, \alpha \cdot Z)$ of the transformed vectors is related to the original by
$$
\mathcal{G}' = M^\top\, \mathcal{G}\, M, \qquad \text{where} \quad
M = \operatorname{diag}\!\left(
\begin{pmatrix} 1 & 0 \\ \alpha_1 & 1 \end{pmatrix},
\ldots,
\begin{pmatrix} 1 & 0 \\ \alpha_n & 1 \end{pmatrix}
\right).
$$
Since $M$ is invertible, $\operatorname{rank}(\mathcal{G}') = \operatorname{rank}(\mathcal{G})$. In particular, the $(m{+}1) \times (m{+}1)$ minors of $\mathcal{G}'$ vanish if and only if the corresponding minors of $\mathcal{G}$ vanish. Thus $\phi_\alpha$ maps $I_{\mathrm{Gram}}$ to itself.
\end{proof}
\begin{remark}
The unipotent group $G$ is non‑reductive. Consequently, Hilbert’s finiteness
theorem does not apply, and the invariant ring $\mathcal{R}^{G}$ need not be
finitely generated (cf.~Nagata’s counterexample to Hilbert’s 14th
problem~\cite{Nagata1959}). In contrast, the field of rational invariants~$\mathcal{F}^{G}$ is finitely generated for any algebraic group.
\end{remark}
As in Proposition~\ref{prop: freeaction},  using the
transformation rules \eqref{eq:transform}, one can verify that the induced action
of $G$ on $\mathcal{R}$ is faithful. Moreover, since a general point
$x\in\mathrm{Spec}(\mathcal{R})$ does not lie on any hypersurface
$\{p_{ij}=0\}$, the action is generically free. We therefore obtain the
following.
\begin{corollary}\label{cor: GenOrbit}
    A general orbit of the $G$-action on $\mathrm{Spec}(\mathcal{R})$ is isomorphic to $\mathbb{G}_a^n$. Thus, 
    $$\dim (G \cdot x) = n \quad \text{for a general}  \;\, x \in \mathrm{Spec}(\mathcal{R}).$$
\end{corollary}
We now introduce the main result from invariant theory that is used in this section.
\begin{definition}
Let $G$ be an algebraic group acting on an irreducible variety $X$ over $\K$.
\begin{itemize}
    \item A rational invariant $f$ is said to \emph{separate} two orbits
$\mathcal{O}_1$ and $\mathcal{O}_2$ if it is defined on both orbits and takes
different values on them.
\item  A set $M$ of invariants separates
two orbits if it contains an invariant  separating
them. 
\item Finally, $M$ \emph{separates orbits in general position} if
there exists a nonempty Zariski-open subset $X_0 \subset X$ such that $M$ separates the
orbits of any two inequivalent points of $X_0$.
\end{itemize}
\end{definition}
\noindent
The next result is due to Rosenlicht with modern exposition presented 
in~\cite[Chapter~2]{PopovVinberg}.

\begin{theorem}[Rosenlicht, 1956] \label{thm: Rosenlicht}
Let $G$ be an algebraic group acting on an irreducible variety~$X$.
There exists a finite set of rational invariants $M = \{f_1,\ldots,f_r\}$ that
separates orbits in general position. Moreover, the field of invariants
$\K(X)^G$ is generated by $M$, and
\begin{equation}\label{eq: trdegInvField}
\mathrm{trdeg}_\K\bigl(\K(X)^G\bigr)
  =
  \dim X - \max_{x\in X}\, \dim(G\cdot x).    
\end{equation}
\end{theorem}
We now recall the definition of the \emph{basic building blocks} \eqref{eq: confblocks intro} from \cite{CPPR}. 
For distinct
indices $i,j,k$ in $[n]=\{1,\ldots,n\}$, these are defined as the following rational functions in $\mathcal{F}$:
\begin{equation}\label{eq: confblocks}
\begin{aligned}
  P_{ij}
  &= p_{ij},\\
  H_{ij}
  &= -2\bigl(r_{ij}\,p_{ij}-q_{ji}\,q_{ij}\bigr),\\
  V_{ijk}
  &
= \frac{q_{ji}\,p_{ik}-q_{ki}\,p_{ij}}{p_{jk}}.
\end{aligned}
\end{equation}
Additionally let
$\mathcal{V}_{ij}$ denote $V_{i,i+1,j}$ with  $j\neq i,i+1$, where indices are taken modulo $n$.
\begin{proposition}\label{prop: inv}
The $n(2n-3)$  functions $P_{ij}$, $H_{ij}$, $\mathcal{V}_{ij}$ are $G$-invariant:
\[
  f(\alpha\cdot x) = f(x)
  \qquad
  \forall\,\alpha\in \mathbb{G}_a^n,\quad x\in \mathrm{Spec}(\mathcal{R}),\quad
  f\in\{P_{ij},\,H_{ij},\,\mathcal{V}_{ij}\}.
\]
\end{proposition}
\begin{proof}
 Invariance of the $P_{ij}$ is immediate. Using $p_{ij}=p_{ji}$ and the transformation rules of
Proposition~\ref{prop:transformation-rules}, the $\alpha$‑dependent terms in
$H_{ij}$ and $\mathcal{V}_{ij}$ cancel identically, which can be verified by direct computation, thus proving invariance of the $H_{ij}$ and $\mathcal{V}_{ij}$.
\end{proof}
\begin{theorem}\label{thm: main}
The field of rational invariants of the action induced by
\eqref{eq:transact} on $\mathrm{Spec}(\mathcal{R})$ is
\[
  \mathcal{F}^{G}
  \;=\;
  \K\Bigl(
    P_{ij}\;\,(i<j),\;\,
    H_{ij}\;\,(i<j),\;\,
    \mathcal{V}_{ij}\;\,(i,j\in[n],\;j\neq i,i+1)
  \Bigr).
\]
\end{theorem}
\begin{proof}
Invariance of the generators was established in
Proposition~\ref{prop: inv}, so the right‑hand side is contained in
$\mathcal{F}^G$. It therefore suffices to show that these invariants separate
general $G$‑orbits. In other words, to show that for any $x,x' \in \mathrm{Spec}(\mathcal R)$ and all admissible indices
\begin{equation}\label{eq: orbitseparators}
   P_{ij}(x')=P_{ij}(x),
  \quad
  H_{ij}(x')=H_{ij}(x),
  \quad
  \mathcal{V}_{ij}(x')=\mathcal{V}_{ij}(x),  
\end{equation}
if and only if  $x'=\alpha\cdot x$ for a unique $\alpha\in \mathbb{G}_a^n$. 
We consider equations~\eqref{eq: orbitseparators} on the dense open subset of $\mathrm{Spec}(\mathcal R)$ where the Gram matrix has maximal rank and the denominators of $\mathcal V_{ij}$ do not vanish.
On this
locus, the relations \eqref{eq: orbitseparators} are not implied by the Gram ideal.
In addition, since $P_{ij}=p_{ij}$, the equality $P_{ij}(x)=P_{ij}(x')$
implies that
\begin{equation}\label{eqn: pij equal}
p'_{ij}=p_{ij} \quad \text{for all} \quad  \;  i\neq j.
\end{equation}
Assuming \eqref{eqn: pij equal}, the equality $\mathcal V_{ij}(x')=\mathcal V_{ij}(x)$ then yields
$
  p_{ij}(q'_{i+1,i}-q_{i+1,i})
  =
  p_{i,i+1}(q'_{ji}-q_{ji})
$ for $j\neq i,i+1$.
That is, the vector
$(q'_{ji}-q_{ji})_{j\neq i}$ is proportional to $(p_{ij})_{j\neq i}$. Setting
\begin{equation*}\label{eq:alphai}
  \alpha_i := \frac{q'_{i+1, i}-q_{i+1, i}}{p_{i,i+1}},
\end{equation*}
it follows that
$
 q'_{ji}=q_{ji}+\alpha_i\,p_{ij}
$. Finally, substituting these expressions into the relations
$H_{ij}(x')=H_{ij}(x)$ gives
$$
  r'_{ij}
  =
  r_{ij}
  + \alpha_i q_{ij}
  + \alpha_j q_{ji}
  + \alpha_i\alpha_j p_{ij},
$$
Thus, by Proposition \ref{prop:transformation-rules} we have $x'=\alpha\cdot x$ for a unique $\alpha\in\mathbb G_a^n$, and the building blocks $P_{ij}, H_{ij}$ and $\mathcal{V}_{ij}$ separate general orbits. The claim now follows from Theorem~\ref{thm: Rosenlicht}.
\end{proof}

As a corollary of Rosenlicht’s theorem, we obtain the dimension of the variety
parametrized by the basic building blocks $P_{ij}$, $H_{ij}$, and $\mathcal{V}_{ij}$. This variety was also studied in~\cite{MSS}.
\begin{corollary}\label{cor: trdegF}
Combining \eqref{eq: trdegF}, \eqref{eq: trdegInvField}, and
Corollary~\ref{cor: GenOrbit}, for $m\ge 5$ we obtain
\begin{equation*}
   \mathrm{trdeg}_\K\bigl(\mathcal{F}^{G}\bigr)
=
   \dim \mathcal{R} - \max_{p \in \mathrm{Spec}(\mathcal{R})}\dim(G \cdot p) 
=\mathrm{trdeg}_\K\bigl(\mathcal{F}\bigr) - n
  =\begin{cases}
      n(2m-4)-{m \choose 2}, & m<2n\\
      n(2n-3), & m\geq 2n.
  \end{cases} 
\end{equation*}
\end{corollary}
\noindent Recall that $\mathrm{trdeg}_\K\bigl(\mathcal{F}^{G}\bigr)$ is the number of algebraically independent $P_{ij}$, $H_{ij}$ and $\mathcal{V}_{ij}$ over $\K$.

\section{Conformal $n$-point functions}\label{sec:ConfNpoint}

As described in Section \ref{sec: preliminaries}, an $n$-point conformal correlator $G_{\mathbf{s},\mathbf{\Delta}}(P_1,Z_1,\ldots,P_n,Z_n)$ with spins $\mathbf{s}=(s_1,\ldots,s_n)\in\mathbb{Z}_{\ge0}^n$ and scaling dimensions $\mathbf{\Delta}=(\Delta_1,\ldots,\Delta_n)\in\mathbb{C}^n$ must satisfy Lorentz invariance \eqref{eq: Lorentz inv}, transversality \eqref{eq: Transv inv}, polynomiality in the entries of $Z_i$, homogeneity~\eqref{eq: homogeneity preliminaries}, null-cone and polarisation constraints \eqref{eq: traceless}. Starting from the perspective that $G_{\mathbf{s},\mathbf{\Delta}}$ could be an arbitrary function of the $P_i$ and $Z_i$, these constraints put strong restrictions on the form that $G_{\mathbf{s},\mathbf{\Delta}}$ can take. 
The goal of this section is to describe the most general form of $G_{\mathbf{s},\mathbf{\Delta}}$ allowed by the constraints~\eqref{eq: Lorentz inv}-\eqref{eq: traceless}. To do this, we first focus on the case of rational functions that satisfy \eqref{eq: Lorentz inv}-\eqref{eq: traceless}. We call such rational functions $F_{\mathbf{s},\mathbf{\Delta}}(P_1,Z_1,\ldots,P_n,Z_n)$, and the construction of their most general form is the result of Theorem~\ref{thm:sec5}. We then comment on how this result constrains the most general form of $n$-point conformal correlators $G_{\mathbf{s},\mathbf{\Delta}}$ in Remark~\ref{rem:sec5generalcase}.
This section is mostly expository, and we review and prove the results of \cite[Section 4.4]{CPPR}.

\medskip

To formulate the main result of this section, we will need the notion of a \textit{cross-ratio}. For a quadruple of distinct indices $\{i,j,k,l\}\subseteq [n]$, the corresponding cross-ratio is defined as 
$$u_{ijkl}(P):=\dfrac{P_{ij}P_{kl}}{P_{ik}P_{jl}}.$$

\begin{remark}
    We emphasize that each $P_{ij}$ is of the form:
    \begin{equation*}
        P_{ij}=P_i\cdot P_j=-P_i^0P_j^0+P_i^1P_j^1+\cdots+P_i^{m-1}P_j^{m-1},
    \end{equation*}
    and so these cross-ratios differ from the cross-ratios on the line $\mathbb{P}^1$ (see e.g. \cite[Section 1]{lammoduli}).
\end{remark}

Theorem \ref{thm: main} indicates that Lorentz invariance, transversality, null-cone constraints and tracelessness mean that each rational function $F_{\mathbf{s},\mathbf{\Delta}}$ satisfying~\eqref{eq: Lorentz inv}-\eqref{eq: traceless} is a function of the variables $P_{ij}$, $H_{ij}$ and $\mathcal{V}_{ij}$. The polynomiality constraint is satisfied by requiring that $F_{\mathbf{s},\mathbf{\Delta}}$ be polynomial in the $H_{ij}$ and $\mathcal{V}_{ij}$. 
The homogeneity constraint \eqref{eq: homogeneity preliminaries} imposes that the terms of this polynomial have a certain fixed multidegree and, as we show in the theorem below, also fixes the dependence on the variables $P_{ij}$.
We can now conclude that the general form of $F_{\mathbf{s},\mathbf{\Delta}}$ is an element of the ring $\mathbb{C}(P_{ij})[H_{ij},\mathcal{V}_{ij}]$ and is thus a sum of monomial terms
\begin{equation}\label{eqn: general F}
    F_{\mathbf{s},\mathbf{\Delta}}(P_1,Z_1,\ldots,P_n,Z_n)=\sum_{k}B_k(P_{ij},H_{ij},\mathcal{V}_{ij})\,,
\end{equation}
where each $B_k$ satisfies the homogeneity property \eqref{eq: homogeneity preliminaries}.

\medskip

We are now ready to formulate and prove the main result of this section.

\begin{theorem} \label{thm:sec5}
    Let $\mathbb{C}(\mathbf{u})$ be the field of rational functions in the cross-ratios $u_{ijkl}(P)$ for all quadruples $\{i,j,k,l\}$.
    Endow the polynomial ring $\mathbb{C}(\mathbf{u})[H_{ij},\mathcal{V}_{ij}]$ with the multigrading given by \eqref{eq: multidegrees VH}. Each rational function $F_{\mathbf{s},\mathbf{\Delta}}$ that satisfies~\eqref{eq: Lorentz inv}-\eqref{eq: traceless} with spins $\mathbf{s}=(s_1,\ldots,s_n)$ and scaling dimensions $\mathbf{\Delta}=(\Delta_1,\ldots,\Delta_n)$ can be written in the form:
    \begin{equation} \label{eq:TheoremSec5}
    \begin{split}
        F_{\mathbf{s},\mathbf{\Delta}}(P_1,Z_1,\ldots,&P_n,Z_n)=\left[\prod_{i<j} P_{ij}^{-c_{ij}}\right]A(\mathbf{u},H_{ij},\mathcal{V}_{ij})\,,\\
        c_{ij}&= \dfrac{\tau_i+\tau_j}{n-2}-\dfrac{1}{(n-1)(n-2)}\sum\limits_{k=1}^n \tau_k\,,
    \end{split}
    \end{equation}
    where each $A\in \mathbb{C}(\mathbf{u})[H_{ij},\mathcal{V}_{ij}]$ is a homogeneous polynomial of multidegree $\mathbf{s}$, and $\tau_i=\Delta_i+s_i$.
\end{theorem}

\begin{remark}\label{rmk: no cross ratios}
    In the case $n=3$, there are no cross-ratios as one cannot pick a quadruple of distinct indices. The field $\mathbb{C}(\mathbf{u})$ in this case is therefore just the field of constants $\mathbb{C}$ and the dependence on the cross-ratios in Theorem \ref{thm:sec5} disappears entirely.
\end{remark}

\begin{proof}
    We start from the assumption that $F_{\mathbf{s},\mathbf{\Delta}}$ is of the form written in equation \eqref{eqn: general F}. Consider first the constraints on the monomials $B_k\in\mathbb{C}(P_{ij})[H_{ij},\mathcal{V}_{ij}]$ imposed by the homogeneity condition \eqref{eq: homogeneity preliminaries}. Note that the variables $P_{ij}$, $H_{ij}$ and $\mathcal{V}_{ij}$ transform as follows under the scalings $P_i\mapsto\lambda_iP_j$ and $Z_i\mapsto\mu_iZ_i$ of all $P_i$ and $Z_i$:
    \begin{equation} \label{eq:scalingPHV}
    P_{ij} \mapsto\lambda_i \lambda_j\, P_{ij}, \quad H_{ij} \mapsto \lambda_i\lambda_j \mu_i \mu_j\, H _{ij}, \quad \mathcal{V}_{ij} \mapsto \lambda_i \mu_i\mathcal{V}_{ij}\,. 
    \end{equation}
    Imposing that the $B_k$ are all homogeneous in the $Z_i$ with degree dictated by \eqref{eq: homogeneity preliminaries} and thus equivalent implies that each of the $B_k$ be homogeneous of multidegree $\mathbf{s}$ in the grading \eqref{eq: multidegrees VH}. 
    Recall from Section \ref{sec:matching_pols} that for a given $\mathbf{s}$, there are a finite number of monomials in $\mathbb{C}[H_{ij},\mathcal{V}_{ij}]$ with this multidegree and so the number of $B_k$ is finite. 
    The dependence of each $B_k$ on the $P_{ij}$ is then constrained by requiring that the $B_k$ have the desired scaling properties under the transformation $P_i\mapsto\lambda_iP_i$. 
    We will now show that the dependence of each monomial $B_k$ on the variables $P_{ij}$ is the same up to cross-ratios.

    \medskip
    
    Consider the coefficient of a single monomial $B_k$ in the field $\mathbb{C}(P_{ij})$. 
    It is a rational function in the variables $P_{ij}$, that is, it takes the form:
    \begin{equation}\label{eqn: rational P}
    \dfrac{\sum\limits_{(a_{ij})\in\mathbb{Z}_{\geq 0}^L}\alpha_{(a_{ij})}\prod\limits_{1\leq i<j\leq n} P_{ij}^{a_{ij}}}{\sum\limits_{(b_{ij})\in\mathbb{Z}_{\geq 0}^L}\beta_{(b_{ij})}\prod\limits_{1\leq i<j\leq n} P_{ij}^{b_{ij}}},
    \end{equation}
    where $L=\binom{n}{2}$.
    Since $B_k$ is homogeneous of multidegree $\mathbf{s}$ in the entries of the $Z_i$, the homogeneity constraint \eqref{eq: homogeneity preliminaries} requires that the function \eqref{eqn: rational P} be homogeneous of multidegree $(-\Delta_1-s_1,\ldots,-\Delta_n-s_n)$ in the $P_i$. Thus, the summands of \eqref{eqn: rational P}, each of which has the form
    \begin{equation*}
        \dfrac{\alpha_{(a_{ij})}\prod\limits_{1\leq i<j\leq n} P_{ij}^{a_{ij}}}{\sum\limits_{(b_{ij})\in\mathbb{Z}_{\geq 0}^L}\beta_{(b_{ij})}\prod\limits_{1\leq i<j\leq n} P_{ij}^{b_{ij}}} = \dfrac{\alpha_{(a_{ij})}}{\sum\limits_{(b_{ij})\in\mathbb{Z}_{\geq 0}^L}\beta_{(b_{ij})}\prod\limits_{1\leq i<j\leq n} P_{ij}^{b_{ij}-a_{ij}}}
    \end{equation*}
    must be homogeneous of this multidegree. Defining $\tilde{c}_{ij}:=b_{ij}-a_{ij}$, this implies that the $\tilde{c}_{ij}$ must satisfy the following system of $n$ linear equations:
    \begin{equation*}
        \sum_{\substack{j=1\\j\neq i}}^n\tilde{c}_{ij}=\Delta_i+s_i\,, \quad i\in\{1,\ldots,n\}\,,
    \end{equation*}
    where we have adopted the labeling convention $\tilde{c}_{ij}=\tilde{c}_{ji}$. Writing $\tilde{\mathbf{c}}$ for the vector of variables $\tilde{c}_{ij}$ and defining $\tau_i:=\Delta_i+s_i$, we find that the linear system is of the form:
    \begin{equation*}
        B\tilde{\mathbf{c}}=(\tau_1,\ldots,\tau_n)^T,
    \end{equation*}
    where $B$ is the $n\times\binom{n}{2}$ vertex-edge incidence matrix of the undirected complete graph $K_n$. A particular solution of this system is given by the vector with entries
    \begin{equation*}
        c_{ij}= \dfrac{\tau_i+\tau_j}{n-2}-\dfrac{1}{(n-1)(n-2)}\sum\limits_{k=1}^n \tau_k.
    \end{equation*}
    The kernel of the matrix $B$ is spanned by the $\binom{n}{4}$ vectors $\tilde{\mathbf{c}}$, each having four non-zero entries: $c_{ij}=c_{kl}=1$ and $c_{il}=c_{jk}=-1$.
    This follows for instance from the results of \cite{toricideals_graphs} applied to the complete graph. Each such vector is the exponent vector of the corresponding cross-ratio~$u_{ijkl}$. This means that the coefficient of each monomial $B_k$ in $\mathbb{C}(P_{ij})$ is a uniquely defined Laurent monomial up to multiplication by an arbitrary rational function in the cross-ratios. We can thus remove the common factor $\prod_{i<j}P_{ij}^{-c_{ij}}$
    from each of the monomials $B_k$ to the front of the sum in \eqref{eqn: general F} to obtain an expression as in formula \eqref{eq:TheoremSec5}.
\end{proof}

\begin{remark} \label{rem:sec5generalcase}
    Note that the cross-ratios satisfy Lorentz invariance, transversality, null-cone and tracelessness constraints, and are invariant under respective rescalings of the $P_i$ and $Z_i$~\eqref{eq:scalingPHV}. 
    In view of these facts, the proof of Theorem \ref{thm:sec5} implies that even if one does not assume that the $n$-point function $G_{\mathbf{s},\mathbf{\Delta}}$ is rational, it can still be written in the form \eqref{eq:TheoremSec5} with $A$ being a polynomial in $H_{ij}$ and $\mathcal{V}_{ij}$ whose coefficients are \emph{arbitrary} functions of the cross-ratios. Following Remark~\ref{rmk: no cross ratios}, one can moreover conclude that three-point functions take exactly the form of $F_{\mathbf{s},\mathbf{\Delta}}$ in Theorem~\ref{thm:sec5} -- they are rational.
\end{remark}

Theorem \ref{thm:sec5} along with the results of Section \ref{sec:matching_pols} (and specifically Theorem \ref{thm: main sec 3}) allows us to compute an upper bound for the dimension of the $\mathbb{C}(\mathbf{u})$-vector space of rational functions $F_{\mathbf{s},\mathbf{\Delta}}$ for any value of $n\geq 3$ and any values of the spins $\mathbf{s}=(s_1,\ldots,s_n)$.
We present the results for $d=3$, $n=3,4$ and certain small values of the $s_i$ in Tables \ref{table:n3} and \ref{table:n4}. The code for this computation and a more comprehensive dataset is available at \cite{zenodo}. 
Note that since the monomials of a given multidegree may be linearly dependent for certain values of $d$, the actual dimension of this space can be lower than the number of homogeneous monomials of this multidegree. 
We illustrate this on a simple example and record the number of algebraically independent monomials for $n=d=3$ in the last row of Table \ref{table:n3}. 
We elaborate on how to obtain these numbers in the next section.

\begin{example} \label{ex:11to10}
    We consider the case $n=d=3$ and $\mathbf{s}=(2,2,2)$. Using Table~\ref{table:n3}, we see that there are $11$ monomials in the $H_{ij}$ and $\mathcal{V}_{ij}$ that have the right multidegree. However, for $n=d=3$ there is a single algebraic relation between the $H_{ij}$ and $\mathcal{V}_{ij}$ which shares this multidegree. This relation is given by
    \begin{equation}
        -2H_{12}H_{23}H_{13}=(\mathcal{V}_{13}H_{23}+\mathcal{V}_{21}H_{13}+\mathcal{V}_{32}H_{12}+2\mathcal{V}_{13}\mathcal{V}_{21}\mathcal{V}_{32})^2
    \end{equation}
    We thus find that the number of algebraically independent monomials with multidegree consistent with the choice $\mathbf{s}=(2,2,2)$ is $10$, not $11$. Note that this relation produces linear dependencies between monomials of multidegree $\tilde{\mathbf{s}}\geq \mathbf{s}$ (the inequality is component-wise).  
\end{example}

\begin{table}[h]
\centering
\small
\setlength{\tabcolsep}{3.5pt}
\renewcommand{\arraystretch}{1.2}

\begin{tabular}{|p{2cm}|c|c|c|c|c|c|c|c|c|c|c|c|c|c|c|c|c|}
\hline
Spins
& 001 & 011 & 111
& 002 & 012 & 112
& 003 & 013 & 113
& 023 & 123 & 222 & 223
& 033 & 133 & 233 & 333 \\
\hline
Monomials
& 1 & 2 & 4
& 1 & 2 & 5
& 1 & 2 & 5
& 3 & 8 &11 & 13
& 4 & 10 & 17 & 23 \\
\hline
Independent monomials& 1 & 2 & 4
& 1 & 2 & 5
& 1 & 2 & 5
& 3 & 8 & $\mathbf{10}$
& $\mathbf{12}$ & 4 & 10 & $\mathbf{15}$ & $\mathbf{19}$
\\
\hline
\end{tabular}

\caption{Number of monomials for given values of spins and $n=d=3$.} \label{table:n3}
\end{table}

\begin{table}[h]
\centering
\small
\setlength{\tabcolsep}{3.5pt} 
\renewcommand{\arraystretch}{1.2}
\begin{tabular}{|c|c|c|c|c|c|c|c|c|c|c|c|c|c|c|}
\hline
Spins 
& 0001 & 0011 & 0111 & 1111 
& 0002 & 0012 & 0112 & 1112 
& 0022 & 0122 & 1122 & 0222 & 1222 & 2222 \\
\hline
Monomials
& 2 & 5 & 14 & 43
& 3 & 8 & 24 & 78
& 14 & 44 & 150 & 85 & 302 & 633 \\
\hline
\end{tabular}
\caption{Number of monomials for given values of spins and $n=4$, $d=3$.}\label{table:n4}
\end{table}

Example \ref{ex:11to10} illustrates that it is helpful to understand the algebraic relations between the building blocks  $P_{ij}, H_{ij}$ and $\mathcal{V}_{ij}$. We discuss these relations in the next section.

\section{Algebraic relations between building blocks}\label{sec: alg rel}

In the previous section we have seen that any $n$-point function can be written as a homogeneous polynomial of multidegree $\mathbf{s}$ in the building blocks $H_{ij}$ and $\mathcal{V}_{ij}$ with coefficients depending on the cross-ratios $u_{ijkl}$. 
This gives an upper bound on the number of independent structures in the $n$-point function: one can just count the number of monomials of multidegree $\mathbf{s}$.
However, these monomials can still be dependent. 
The goal of this section is to investigate dependencies between them. In other words, we study algebraic relations between the basic building blocks $P_{ij}$, $H_{ij}$ and $\mathcal{V}_{ij}$. 

\medskip

As explained in the previous section, the motivation for studying these relations is that the monomial counts of Section~\ref{sec:matching_pols} are formal counts: they treat the building blocks as algebraically independent. In fixed spacetime dimension this need not be true, because Gram constraints among the vectors $P_i$ and $Z_i$ can produce algebraic relations among $P_{ij}$, $H_{ij}$, and $\mathcal V_{ij}$. These relations reduce the number of genuinely independent conformal structures and determine the coordinate ring whose Hilbert function gives the dimension-dependent count.

\medskip

The building blocks $P_{ij}$, $H_{ij}$ and $\mathcal{V}_{ij}$ indeed satisfy relations implied by the Gram constraints on the vectors $P_1,\ldots,P_n,Z_1,\ldots,Z_n$. 
These relations are the central topic of \cite{MSS}, and Example \ref{ex:11to10} treats the simplest case of this situation. 
Our setup differs slightly from that of \cite{MSS} in that we do not consider the full set of the variables $V_{i,jk}$ but only a subset consisting of $n(n-2)$ many  $\mathbb{C}(\mathbf{u})$-linearly independent variables $\mathcal{V}_{ij}$ suggested in \cite[Eqn. (4.75)]{CPPR}. 
Our choice of variables is more natural: in fact, the only relations between $P_{ij}, H_{ij}$ and $\mathcal{V}_{ij}$ are those implied by the Gram constraints, as the following proposition demonstrates. 

\begin{proposition}\label{prop:PHVrels}
    Let $m=d+2\ge 5$. If all the inner products $P_i\cdot P_j$, $P_i\cdot Z_j$ and $Z_i\cdot Z_j$ are $\mathbb{C}$-algebraically independent, i.e., if $m\geq 2n$, then so are all the building blocks $P_{ij}$, $H_{ij}$, $\mathcal{V}_{ij}$. 
\end{proposition}

\begin{proof}
    If the inner products are algebraically independent, which by a result of Weyl is the case exactly when $m\geq 2n$, then by Corollary~\ref{cor: trdegF}, the transcendence degree of the field generated by the building blocks, and thus, the dimension of the corresponding variety, is equal to the number of non-zero inner products minus $n$. 
    The number of inner products is $2\binom{n}{2}+n(n-1)$. Subtracting $n$, we get $n(2n-3)$. 
    This is exactly the number of building blocks $P_{ij}, H_{ij}, \mathcal{V}_{ij}$, which means they are algebraically independent.    
\end{proof}

Our next result generalises the second statement in \cite[Lemma 3.7]{MSS}.
It allows to rewrite any relation between $P_{ij}$, $H_{ij}$ and $V_{i,jk}$ in terms of $P_{ij}$, $H_{ij}$ and~$\mathcal{V}_{ij}$.

\begin{lemma} \label{lem:linrels_V}
For any set of indices $i,j,k,l$ the following linear relation holds for the variables $V_{i,jk}$ for arbitrary $d$: 
$$V_{i,jk}P_{il}P_{jk}-V_{i,jl}P_{ik}P_{jl}+V_{i,kl}P_{ij}P_{kl}=0.$$
In particular, setting $j=i+1$ one gets 
$$\dfrac{P_{ik}P_{i+1,l}}{P_{i,i+1}P_{kl}}\mathcal{V}_{il}-\dfrac{P_{il}P_{i+1,k}}{P_{i,i+1}P_{kl}}\mathcal{V}_{ik} = V_{i,kl}.$$    
\end{lemma}
\begin{proof}
   One can directly check these identities by substituting in the parametrization \eqref{eq: confblocks}.
\end{proof}

Let $\mathcal{V}_{1,n,m}$ be the variety of symmetric $n\times n$ matrices of rank at most $m$ with zeros along the diagonal, as in \cite[Section 2]{MSS}.
This variety is naturally isomorphic to the variety parametrised by the building blocks $P_{ij}$.
In the following result we turn to the variety parametrised by the cross-ratios $u_{ijkl}$ and compute its dimension, which is equal to~$\mathrm{trdeg}_{\mathbb{C}}\mathbb{C}(\mathbf{u})$.

\begin{proposition} \label{prop:cu}
Let $\mathbb{C}(\mathbf{u})$ be the field of rational functions in the cross-ratios $u_{ijkl}$ for $\{i,j,k,l\}\subset [n]$. 
Then $\mathrm{trdeg}_{\mathbb{C}}\mathbb{C}(\mathbf{u})=\mathrm{dim}(\mathcal{V}_{1,n,m})-n$. For $m\ge 3$ (any $d\in\mathbb{N}$) and $n>m$,
\begin{equation}
    \mathrm{trdeg}_{\mathbb{C}}\mathbb{C}(\mathbf{u})=n(m-2)+\binom{m}{2}\,.
\end{equation}
For $n\leq m$, we have $\mathrm{trdeg}_{\mathbb{C}}\mathbb{C}(\mathbf{u})=\binom{n}{2}-n$.
\end{proposition}
\begin{proof}
 The variety $\mathcal{V}_{1,n,m}$ is naturally parametrised by the building blocks $P_{ij}$ via the matrix
    \begin{equation}
        \begin{pmatrix}
            0&P_{12}&P_{13}&\cdots&P_{1n}\\P_{12}&0&P_{23}&\cdots&P_{2n}\\\vdots&\vdots&\vdots&\ddots&\vdots\\P_{1n}&P_{2n}&P_{3n}&\cdots&0
        \end{pmatrix},
    \end{equation}
    where $P_{ij}=P_i\cdot P_j$ and $P_i\in\mathbb{K}^m$ for all $i$. 
    The algebraic torus $(\mathbb{C}^*)^n$ acts on $\mathbb{C}^{nm}$ by
    \begin{equation}
        \forall \ \lambda=(\lambda_1,\ldots,\lambda_n)\in (\mathbb{C}^*)^n: \quad \lambda\cdot P_i=\lambda_i P_i\,, \quad i=1,\ldots,n\,.
    \end{equation}
    This $(\mathbb{C}^*)^n$-action can be lifted from $\mathbb{C}^{nm}$ to $\mathcal{V}_{1,n,m}$ through
    \begin{equation}
        M\mapsto\mathrm{diag}(\lambda_1,\ldots,\lambda_n)\cdot M\cdot\mathrm{diag}(\lambda_1,\ldots,\lambda_n).
    \end{equation}
    Clearly, for generic $M\in\mathcal{V}_{1,n,m}$, the dimension of the $(\mathbb{C}^*)^n$-orbit of $M$ is $n$. 
    Furthermore, the field of invariants of $\mathbb{C}(P_{ij})$ under this torus action is $\mathbb{C}(\mathbf{u})$ (this follows e.g. from \cite[Section 1.4]{sturmfelsInvTh}).
    Thus, by Theorem \ref{thm: Rosenlicht},
    \begin{equation}
        \mathrm{trdeg}_{\mathbb{C}}\mathbb{C}(\mathbf{u})=\mathrm{dim}(\mathcal{V}_{1,n,m})-n.
    \end{equation}
    By \cite[Theorem 2.5]{MSS}, for $3\leq m<n$ we have $\mathrm{dim}(\mathcal{V}_{1,n,m})=n(m-1)-\binom{m}{2}$ and thus
    \begin{equation}
        \mathrm{trdeg}_{\mathbb{C}}\mathbb{C}(\mathbf{u})=n(m-2)-\binom{m}{2}.
    \end{equation}
    When $m\geq n$, the building blocks $P_{ij}$ are $\mathbb{C}$-algebraically independent and the dimension of the variety $\mathcal{V}_{1,n,m}$ is simply the number of $P_{ij}$, which is $\binom{n}{2}$. 
\end{proof}

\begin{lemma} \label{lem:fieldext}
    Consider the following field extensions:
    \[
\begin{tikzcd}
K:=\mathbb{C}(\mathbf{u}) \arrow[r, phantom, "\subset"] 
  \arrow[d, phantom, "\rotatebox{270}{$\subset$}"'] 
&
A:=K(P_{ij}) =\mathbb{C}(P_{ij})\arrow[d, phantom, "\rotatebox{270}{$\subset$}"] \\
L:=K(H_{ij},\mathcal{V}_{ij}) \arrow[r, phantom, "\subset"] 
&
F:=K(H_{ij},\mathcal{V}_{ij})(P_{ij}) = \mathbb{C}(P_{ij},H_{ij},\mathcal{V}_{ij})
\end{tikzcd}
\]
We have $\mathrm{trdeg}_K A = n$, and $\mathrm{trdeg}_KL=\mathrm{trdeg}_A F$. 
Therefore, we also have $\mathrm{trdeg}_LF=n$.
\end{lemma}
\begin{proof}
The first statement follows from Proposition \ref{prop:cu}: we have $\mathrm{trdeg}_\mathbb{C}A=\mathrm{dim}(\mathcal{V}_{1,n,m})$ and $\mathrm{trdeg}_\mathbb{C} K = \mathrm{dim}(\mathcal{V}_{1,n,m}) - n$.
To prove the second statement, it suffices to show that any algebraic relation between $H_{ij}$ and $\mathcal{V}_{ij}$ over $\mathbb{C}(P_{ij})$ is in fact an algebraic relation over~$\mathbb{C}(\mathbf{u})$. 
To show this, recall that the ring $\mathbb{C}[P_{ij}, H_{ij},\mathcal{V}_{ij}]$ admits a natural $\mathbb{Z}^{2n}$ multigrading as in~\eqref{eq:scalingPHV}. 
Any algebraic relation between $P_{ij}, H_{ij}$ and $\mathcal{V}_{ij}$ has to be homogeneous with respect to this multigrading, since these blocks are functions of the vectors $P_i$ and $Z_i$ whose entries are chosen independently, and the multigrading records the dependence on $P_i$ and $Z_i$. 
Since the last $n$ entries of the multidegree vector $\mathrm{deg}(P_{ij})$ are zero for all $P_{ij}$, the coefficients (that are functions in $P_{ij}$) at all monomials in $H_{ij}$ and $\mathcal{V}_{ij}$ in such a relation have to have the same multidegree.
Dividing by one of these coefficients, we get a relation in $\mathbb{C}(P_{ij})[H_{ij}, \mathcal{V}_{ij}]$ whose coefficients have multidegree zero. 
The field of rational functions in $P_{ij}$ of multidegree zero is the field of invariants of the $(\mathbb{C}^*)^n$-action from Proposition \ref{prop:cu} and is generated by the cross-ratios $\mathbf{u}$, so this relation is indeed a relation in $\mathbb{C}(\mathbf{u})[H_{ij},\mathcal{V}_{ij}]$.
\end{proof}

The following is the main result of this section. We count the number of building blocks $H_{ij}$ and $\mathcal{V}_{ij}$ that are algebraically independent over $\mathbb{C}(\mathbf{u})$. 
\begin{theorem}
Let $m\geq 5$ and let $K$ and $L$ be as in Lemma \ref{lem:fieldext}.
Then we have
\begin{equation}
\mathrm{trdeg}_K L=
\begin{cases}
   n(m-2), \quad &n>m,\\
   n(2m-4)-\binom{m}{2} - \binom{n}{2}, \quad &2n>m>n,\\
   \binom{n}{2}+n(n-2), \quad &m>2n.
\end{cases}
\end{equation}
\end{theorem}
\begin{proof}
    Consider the following tower of field extensions:
    \begin{equation}
        \mathbb{C}\subset K \subset L \subset F:=\mathbb{C}(P_{ij},H_{ij},\mathcal{V}_{ij}).
    \end{equation}
We have $\mathrm{trdeg}_K L= \mathrm{trdeg}_\mathbb{C}F - \mathrm{trdeg}_\mathbb{C} K-\mathrm{trdeg}_L F$. 
By Corollary \ref{cor: trdegF},
\begin{equation}
\mathrm{trdeg}_\mathbb{C}F =n(2m-4)-\binom{m}{2} \quad \mathrm{for} \quad m<2n.
\end{equation}
By Proposition \ref{prop:cu}, $\mathrm{trdeg}_\mathbb{C} = \mathrm{dim}(\mathcal{V}_{1,n,m})-n$ and by Lemma \ref{lem:fieldext}, we have $\mathrm{trdeg}_LF=n$.
Combining these equalities yields the result. 
In the case $m\geq 2n$, the building blocks are independent by Proposition \ref{prop:PHVrels}.
\end{proof}

Let $I$ be the ideal of the polynomial ring $\mathbb{C}(\mathbf{u})[H_{ij},\mathcal{V}_{ij}]$ recording all algebraic dependencies over $\mathbb{C}(\mathbf{u})$ between the building blocks $H_{ij}$ and $\mathcal{V}_{ij}$. 
In addition, let $\tilde{I}$ in $\mathbb{C}[H_{ij},\mathcal{V}_{ij}]$ be the ideal encoding all $\mathbb{C}$-algebraic dependencies between the $H_{ij}$ and $\mathcal{V}_{ij}$.
The number of algebraically independent over $\mathbb{C}(\mathbf{u})$ (respectively $\mathbb{C}$) monomials of multidegree $\mathbf{s}$ is counted by the Hilbert function $h_I(\mathbf{s})$ (respectively~$h_{\tilde{I}}(\mathbf{s})$) and we have $h_I(\mathbf{s})\leq h_{\tilde{I}}(\mathbf{s})$. 
Since the cross-ratios $u_{ijkl}$ depend on the same variables as the building blocks $H_{ij}$ and $\mathcal{V}_{ij}$, computing $h_I(\mathbf{s})$ is an algorithmically challenging task.
We provide code to compute $h_{\tilde{I}}(\mathbf{s})$ using numerical methods, giving an upper bound for $h_I(\mathbf{s})$. The code along with the description of the algorithm is available at \cite{zenodo}. For $n=3$ this upper bound is in fact an equality (see the last row of Table \ref{table:n3} for the values of this Hilbert function), and for $n=4$ it improves significantly over the upper bound given by the methods from Section \ref{sec:matching_pols}. This is summarized in Table~\ref{table:6}. In Table~\ref{table:7} we present results for $n=d=4$. Note that for $n=3$, $d=4$ all the building blocks are algebraically independent and so the monomial count provided by Theorem~\ref{thm: main sec 3} is the number of independent monomials. A larger dataset with more spin configurations is available at \cite{zenodo}.

\begin{table}[h]
\centering
\small
\setlength{\tabcolsep}{3.5pt} 
\renewcommand{\arraystretch}{1.2}
\begin{tabular}{|p{2cm}|c|c|c|c|c|c|c|c|c|c|c|c|c|c|}
\hline
Spins 
& 0001 & 0011 & 0111 & 1111 
& 0002 & 0012 & 0112 & 1112 
& 0022 & 0122 & 1122 & 0222 & 1222 & 2222 \\
\hline
Monomials
& 2 & 5 & 14 & 43
& 3 & 8 & 24 & 78
& 14 & 44 & 150 & 85 & 302 & 633
\\
\hline
Independent Monomials
& 2 & 5 & 14 & $\mathbf{41}$
& 3 & 8 & $\mathbf{23}$ & $\mathbf{68}$
& $\mathbf{13}$ & $\mathbf{38}$ & $\mathbf{113}$ & $\mathbf{63}$ & $\mathbf{188}$ & $\mathbf{313}$ \\
\hline
\end{tabular}
\caption{Number of independent monomials for given values of spins and $n=4$, $d=3$, taking algebraic relations from the ideal $\tilde{I}$ into account.}\label{table:6}
\end{table}

\begin{table}[h]
\centering
\small
\setlength{\tabcolsep}{3.5pt} 
\renewcommand{\arraystretch}{1.2}
\begin{tabular}{|p{2cm}|c|c|c|c|c|c|c|c|c|c|c|c|c|c|}
\hline
Spins 
& 0001 & 0011 & 0111 & 1111 
& 0002 & 0012 & 0112 & 1112 
& 0022 & 0122 & 1122 & 0222 & 1222 & 2222 \\
\hline
Monomials
& 2 & 5 & 14 & 43
& 3 & 8 & 24 & 78
& 14 & 44 & 150 & 85 & 302 & 633
\\
\hline
Independent Monomials
& 2 & 5 & 14 & 43
& 3 & 8 & 24 & 78
& 14 & 44 & $\mathbf{149}$ & $\mathbf{84}$ & $\mathbf{294}$ & $\mathbf{594}$ \\
\hline
\end{tabular}
\caption{Number of independent monomials for given values of spins and $n=4$, $d=4$, taking algebraic relations from the ideal $\tilde{I}$ into account.}\label{table:7}
\end{table}

\section{Bose symmetry}\label{sec: Bose}
Physical systems are invariant under the exchange of identical particles. 
At the level of correlators $G_{\mathbf{s},\mathbf{\Delta}}(P_1,Z_1,\ldots,P_n,Z_n)$ involving bosonic fields, this fact manifests itself in the form of Bose symmetry: if for some $i\neq j$, $(s_i,\Delta_i)=(s_j,\Delta_j)$, the corresponding correlator $G_{\mathbf{s},\mathbf{\Delta}}$ is invariant under the permutation of the pairs of variables $(P_i,Z_i)$ and $(P_j,Z_j)$, in other words,~\eqref{eq:bose} holds. For a given physical theory, this symmetry is built into how correlators are defined. However, from the bootstrap perspective Bose symmetry instead provides an additional set of constraints on the general form of certain correlators. From this perspective, for each $i\neq j$ with $(s_i,\Delta_i)=(s_j,\Delta_j)$ we obtain a system of linear constraints on the coefficients of monomials that can appear in a given $n$-point function.
We now describe these constraints using the language of invariant theory.

\medskip

Let $S_n$ be the symmetric group on $n$ elements. A permutation $\sigma\in S_n$ acts on the polynomial ring $\mathbb{C}[P_{ij}, H_{ij}, V_{i,jk}]$ via 
\begin{equation} \label{eq:BosePHV}
    P_{ij} \mapsto P_{\sigma(i)\sigma(j)},\quad
    H_{ij}\mapsto H_{\sigma(i)\sigma(j)},\quad
    V_{i,jk}\mapsto V_{\sigma(i),\sigma(j)\sigma(k)}.
\end{equation}

We now wish to define the induced action of $S_n$ on the ring $\mathbb{C}(\mathbf{u})[H_{ij},\mathcal{V}_{ij}]$. For the cross ratios $u_{ijkl}$ this is simple, we have $\sigma(u_{ijkl})=u_{\sigma(i)\sigma(j)\sigma(k)\sigma(l)}$. For the $\mathcal{V}_{ij}$ we have the slight complication that, generally
\begin{equation}
    \sigma\cdot \mathcal{V}_{ij}=\sigma\cdot V_{i,i+1 j}=V_{\sigma(i), \sigma(i+1)\sigma(j)}\neq V_{\sigma(i), \sigma(i)+1\sigma(j)}=\mathcal{V}_{\sigma(i)\sigma(j)}\,.
\end{equation} 
Despite this, the action of $\sigma$ on $\mathcal{V}_{ij}$ can still easily be described using only the variables $\mathcal{V}_{ij}$ and cross-ratios in the $P_{ij}$. Namely, applying Lemma \ref{lem:linrels_V}, we obtain
\begin{equation}\label{eq:Bose_CalV}
    \sigma\cdot \mathcal{V}_{ij} = V_{\sigma(i), \sigma(i+1)\sigma(j)} =\dfrac{P_{\sigma(i)\sigma(i+1)}P_{\sigma(i)+1\sigma(j)}}{P_{\sigma(i)\sigma(i)+1}P_{\sigma(i+1)\sigma(j)}}\mathcal{V}_{\sigma(i)\sigma(j)} - \dfrac{P_{\sigma(i)\sigma(j)}P_{\sigma(i)+1\sigma(i+1)}}{P_{\sigma(i)\sigma(i)+1}P_{\sigma(i+1)\sigma(j)}}\mathcal{V}_{\sigma(i)\sigma(i+1)}.
\end{equation}
Since both coefficients on the right-hand side are cross-ratios, we again land in $\mathbb{C}(\mathbf{u})[H_{ij}, \mathcal{V}_{ij}]$.
Note that if $\sigma(i+1)=\sigma(i)+1$ for some $i$, the above simplifies to $\sigma\cdot \mathcal{V}_{ij}=\mathcal{V}_{\sigma(i)\sigma(j)}$.

\medskip

We can now define Bose symmetry as invariance under the action of a subgroup of $S_n$ on the polynomial ring $\mathbb{C}(\mathbf{u})[H_{ij}, \mathcal{V}_{ij}]$ given by Equations \eqref{eq:BosePHV} and \eqref{eq:Bose_CalV}. 
Namely, consider the subgroup $\mathfrak S$ of $S_n$ generated by transpositions $(ij)$ for all $i,j$ such that $(s_i,\Delta_i)=(s_j, \Delta_j)$. Then Bose symmetry dictates that the rational part of an $n$-point function is an element of the ring of invariants
$\mathbb{C}(\mathbf{u})[H_{ij}, \mathcal{V}_{ij}]^\mathfrak{S}$. 
Computing this ring of invariants for arbitrary $n$ is a difficult, albeit algorithmic \cite{Derksen} problem. 
In \cite{GHL} the authors studied a similar problem, exploring the consequences of Bose symmetry for polynomials in the $P_{ij}$, allowing for $P_{ii}\neq0$. Another physically-motivated paper that investigates rings of invariants under both the action of an orthogonal group and a permutation group is \cite{HLMM}.

\medskip

We now concentrate on the case of three-point functions and treat $H_{ij}$ and $\mathcal{V}_{ij}$ as formal variables. This means that we do not take the algebraic relations between them into account. By Remarks~\ref{rmk: no cross ratios} and~\ref{rem:sec5generalcase},  three-point functions are rational and do not depend on cross-ratios, so they are elements of the ring $\mathbb{C}[H_{ij},\mathcal{V}_{ij}]$. Bose symmetry implies that they are in fact elements of the ring of invariants $\mathbb{C}[H_{ij},\mathcal{V}_{ij}]^{\mathfrak{S}}$. We present formulas for the dimension of the space of monomials that can appear in a three-point function that has Bose symmetry of some kind. Note that this question was addressed in \cite[Section 2.5.1]{KSD} from a representation theory perspective. We offer an explicit formula based on elementary methods. It would be interesting to generalize these results to higher~$n$.

\medskip

For simplicity of notation, we set $V_1:=V_{1,23}$, $V_2:=V_{2,31}$ and $V_3:=V_{3,12}$. By Theorem \ref{thm:sec5}, a three-point function $G_{\mathbf{s},\mathbf{\Delta}}$ is then, up to multiplication by a Laurent monomial in the $P_{ij}$ that is invariant under all transpositions, a polynomial of the form $\sum A_{(a,b,c,d,e,f)}V_1^aV_2^bV_3^cH_{12}^dH_{23}^eH_{13}^f$, where the coefficients $A$ are $\mathbb{C}$-valued and where the exponents have degree $(s_1,s_2,s_3)$, i.e.,
\begin{equation}\label{eqn: exponent relations}
    a+d+f=s_1\,,\quad b+d+e=s_2\,,\quad c+e+f=s_3\,.
\end{equation}

We first treat the case in which all spins and scaling dimensions are equal.
\begin{proposition} \label{prop:Bose_threespins}
Suppose $(s_1,\Delta_1)=(s_2,\Delta_2)=(s_3,\Delta_3)=(s,\Delta)$. Then the dimension of the vector space of three-point functions satisfying Bose symmetry, ignoring the algebraic dependencies between monomials of multidegree $(s,s,s)$, is 
\begin{equation}
   \frac{1}{4}\binom{s+4}{3} \quad \text{for s even and } \quad \frac{1}{24}s(s-1)(s+1) \quad \text{for s odd.}
\end{equation}
\end{proposition}
\begin{proof}
   Bose symmetry imposes the following constraints on each coefficient $A_{(a,b,c,d,e,f)}$ of the monomial $V_1^aV_2^bV_3^cH_{12}^dH_{23}^eH_{13}^f$:
    \begin{equation*}
    \begin{split}
        A_{(a,b,c,d,e,f)}=(-1)^{a+b+c}&A_{(b,a,c,d,f,e)}\,,\quad A_{(a,b,c,d,e,f)}=(-1)^{a+b+c}A_{(c,b,a,e,d,f)}\,,\\
        &A_{(a,b,c,d,e,f)}=(-1)^{a+b+c}A_{(a,c,b,f,e,d)}\,.
    \end{split}
    \end{equation*}
    Among the constraints written above, the number of them which are independent depends on how many of $d$, $e$ and $f$ are equal. Consequently, to understand by how much these constraints reduce the dimension of the vector space of monomials, we first need to know the number of monomials with; $d=e=f$, any two of $d$, $e$ and $f$ equal, or none of $d$, $e$ and $f$ equal. First recall that due to~\eqref{eqn: exponent relations}, the non-negative numbers $d$, $e$ and $f$ must satisfy the following system of inequalities
    \begin{equation*}
        d+e\leq s\,,\quad e+f\leq s\,,\quad d+f\leq s\,.
    \end{equation*}
    When $d=e=f$ this reduces to $2d\le s$, which for a given $s$ this has $\left\lfloor s/2 \right\rfloor+1$ many non-negative integer solutions. This is the number of coefficients $A$ with $d=e=f$. Now suppose $d=e\neq f$. Then the system reduces to
    \begin{equation*}
            2d\leq s\,,\quad f\leq s-d\,.
    \end{equation*}
    For a given $s$, the number of non-negative solutions to this system is 
    $$\left(s-\left\lfloor\dfrac{s}{2}\right\rfloor+1\right)+\ldots+(s+1)=\left(\left\lfloor\dfrac{s}{2}\right\rfloor + 1\right)\left(s+1-\frac{1}{2}\left\lfloor\frac{s}{2}\right\rfloor\right)\,.$$
    Subtracting the number of solutions with $d=e=f$, we arrive at $(\lfloor s/2\rfloor + 1)(s-\frac{1}{2}\lfloor s/2\rfloor)$. Since the number of solutions for the cases $d\neq e=f$ and $d=f\neq e$ is the same, we find that the total number of coefficients with any two of $d$, $e$ and $f$ equal is $3(\lfloor s/2\rfloor + 1)(s-\frac{1}{2}\lfloor s/2\rfloor)$. Finally, using the above counts and Proposition~\ref{prop: Ns1s2s3}, the number of coefficients with none of $d$, $e$ and $f$ equal is given by
    \begin{equation*}
        N(s,s,s)-3\left(\left\lfloor\dfrac{s}{2}\right\rfloor+1\right)\left(s-\frac{1}{2}\left\lfloor\frac{s}{2}\right\rfloor\right)-\left\lfloor \dfrac{s}{2}\right\rfloor-1.
    \end{equation*}
    We now return to the constraints. First note that, using~\eqref{eqn: exponent relations}, $a+b+c=3s-2(d+e+f)$. Consequently, for $s$ odd, the relations satisfied by the coefficients $A$ above result in all of the coefficients with any of $d$, $e$ or $f$ equal being set to zero. What remains to be considered are the constraints among coefficients with none of $d$, $e$ and $f$ equal. In this case, there are five independent constraints, leading to these coefficients coming in groups of six. Thus, for $s$ odd, the dimension of the space of monomials reduces to
    \begin{equation*}
        \frac{1}{6}\left[N(s,s,s)-3\left(\left\lfloor\dfrac{s}{2}\right\rfloor+1\right)\left(s-\frac{1}{2}\left\lfloor\frac{s}{2}\right\rfloor\right)-\left\lfloor \dfrac{s}{2}\right\rfloor-1\right]=\frac{1}{24}s(s-1)(s+1).
    \end{equation*}
    Now we consider the case of $s$ even. In this case, none of the constraints on the coefficients $A$ set any of them equal to zero. When $d=e=f$ there are no constraints on the coefficients. When any two of $d$, $e$ and $f$ are equal, there are two independent constraints leading to triplets that are equal up to a sign. Thus, for $s$ even 
    the total dimension of the space of monomials reduces to
    \begin{equation*}
        \left\lfloor\dfrac{s}{2}\right\rfloor+1+\left(\left\lfloor\dfrac{s}{2}\right\rfloor+1\right)\left(s-\frac{1}{2}\left\lfloor\dfrac{s}{2}\right\rfloor\right)+\frac{1}{6}\left[N(s,s,s)-3\left(\left\lfloor\dfrac{s}{2}\right\rfloor+1\right)\left(s-\frac{1}{2}\left\lfloor\frac{s}{2}\right\rfloor\right)-\left\lfloor \dfrac{s}{2}\right\rfloor-1\right].
    \end{equation*}
    Since $s$ is even, this simplifies to $\frac{1}{4}\binom{s+4}{3}$.
\end{proof}

\begin{remark}
    The dimension count of Proposition~\ref{prop:Bose_threespins} is given by sequence A006918 in OEIS~\cite{oeis} for $n=s+1$, and admits a number of combinatorial interpretations.
\end{remark}

It remains to consider the case of only two of the spins and scaling dimensions equal.
\begin{proposition}\label{prop:Bose-two-spins}
    Let $(s_1,\Delta_1)=(s_2,\Delta_2)=(s,\Delta)$ and $s<s_3$. Then the dimension of the vector space of three-point functions satisfying Bose symmetry, ignoring the algebraic dependencies between monomials of multidegree $(s,s,s_3)$, is
     $$\frac{1}{2}[N(s,s,s_3)+(-1)^{s_3}N_1(s,s_3)]\,,\quad\mathrm{where}\quad N_1(s,s_3)=\left(\left\lfloor\dfrac{s_3}{2}   \right\rfloor+1\right)\left(s+1-\frac{1}{2}\left\lfloor\dfrac{s_3}{2}\right\rfloor\right)\,,$$
      and where $N(s,s,s_3)$ is the number of lattice points in the conformal polytope from Proposition~\ref{prop: Ns1s2s3} for $s_1=s_2=s$.
     For $(s_2,\Delta_2)=(s_3,\Delta_3)=(s,\Delta)$ and $s_1<s$ this dimension~is
      $$\frac{1}{2}[N(s_1,s,s)+(-1)^{s_1}N_1(s,s_1)]\,.$$
\end{proposition}
\begin{proof}
     When only two spins are equal, Bose symmetry produces just one family of constraints
    $$A_{(a,b,c,d,e,f)}=(-1)^{a+b+c}A_{(b,a,c,d,f,e)}\,.$$
    Let $N(s,s,s_3)$ denote the number of non-negative solutions of the system 
    $$a+d+f=s\,,\quad b+d+e=s\,,\quad c+e+f=s_3\,.$$
    This is the number $N(s_1,s_2,s_3)$ from Proposition \ref{prop: Ns1s2s3}, specialized to the case $s_1=s_2=s$. For a given choice of $s$ and $s_3$, let $N_1(s,s_3)$ be the number of non-negative integer solutions with $e=f$ (and hence $a=b$). As in the proof of Proposition~\ref{prop:Bose_threespins}, we find that
    $$N_1(s,s_3)=\left(\left\lfloor\dfrac{s_3}{2}\right\rfloor+1\right)\left( s+1-\frac{1}{2}\left\lfloor\frac{s_3}{2}\right\rfloor\right),$$
    and the dimension of the vector space of three-point functions satisfying Bose symmetry is
    $$\frac{1}{2}[N(s,s,s_3)+(-1)^{s_3}N_1(s,s_3)]\,.$$
    The argument for the case $(s_2,\Delta_2)=(s_3,\Delta_3)=(s,\Delta)$ and $s_1<s$ is entirely analogous. 
\end{proof}

To conclude this section, we present a table containing the dimension of the space of three-point functions respecting Bose symmetry. The code we used to obtain it is available at \cite{zenodo} and takes into account algebraic dependencies between the building blocks. Thus, the counts in Table \ref{table:bose} are slightly different from the ones given by Propositions \ref{prop:Bose_threespins} and \ref{prop:Bose-two-spins}. 

\begin{table}[h]
\centering
\small
\setlength{\tabcolsep}{3.5pt}
\renewcommand{\arraystretch}{1.2}

\begin{tabular}{|p{2cm}|c|c|c|c|c|c|c|c|c|c|c|c|c|c|c|c|c|}
\hline
Spins
& 001 & 011 & 111
& 002 & 012 & 112
& 003 & 013 & 113
& 023 & 123 & 222 & 223
& 033 & 133 & 233 & 333 \\
\hline
Independent monomials& 1 & 2 & 4
& 1 & 2 & 5
& 1 & 2 & 5
& 3 & 8 & 10
& 12 & 4 & 10 & 15 & 19
\\
\hline
Bose & 0 & 2 & 0 & 1 &
2 & 4 & 0 & 2 & 1 &3 & 8 & 4
& 4 & 4 &3 & 10 &1\\
\hline
\end{tabular}

\caption{Number of structures respecting Bose symmetry for $n=d=3$.} \label{table:bose}
\end{table}

\section{Partial conservation}\label{sec: conservation}
\noindent Beyond conformal covariance and Bose symmetry, certain $n$-point functions must satisfy additional constraints imposed by partial conservation.
These require that the $n$-point correlator be annihilated by a specific set of differential operators. 
These differential operators are defined in terms of the embedding-space vectors $P_i$ and $Z_i$. 
More precisely, suppose the $n$-point function $G_{\mathbf{s},\mathbf{\Delta}}$ has $\Delta_i=d-1+t_i$ for some $i\in\{1,\ldots,n\}$ and $t_i \in \{0, \dots, s_i-1\}$. 
Then, following \cite{DNW}, it must satisfy 
\begin{equation}\label{eqn: partial conservation}
    \left(\frac{\partial}{\partial P_i}\cdot D_{Z_i}\right)^{s_i-t_i}
    G_{\mathbf{s},\mathbf{\Delta}}(P_1,Z_1,\ldots,P_n,Z_n)=0,
\end{equation}
where $D_{Z}$ denotes the vector of differential operators given by
\begin{equation}\label{eqn:ThomasTodorov}
    D_{Z}
    =
    \left(\frac{d}{2}-1+Z\cdot\frac{\partial}{\partial Z}\right)\frac{\partial}{\partial Z}
    -\frac{1}{2}Z\frac{\partial^2}{\partial Z\cdot\partial Z}\,.
\end{equation}
If there are multiple $i\in\{1,\ldots,n\}$ for which this is the case, then $G_{\mathbf{s},\mathbf{\Delta}}$ must satisfy \eqref{eqn: partial conservation} for each such $i$. When bootstrapping $G_{\mathbf{s},\mathbf{\Delta}}$, these constraints serve as restrictions on the space of structures that can appear in $G_{\mathbf{s},\mathbf{\Delta}}$. This approach was applied to three-point functions in \cite{DanielNathan} to construct no-go theorems for the existence of certain types of particles in de Sitter space. The authors derived these constraints on a case-by-case basis for specific choices of spins $s_i$, scaling dimensions $\Delta_i$ and spatial dimension $d$. 
To avoid this case-by-case analysis, one can approach the problem from a different perspective. Instead, one can ask for the most general solution to the differential equation
$(\partial_{P_i}\cdot D_{Z_i})^{s_i - t_i}G_{\mathbf{s}, \mathbf{\Delta}} = 0$ that satisfies the constraints~\eqref{eq: Lorentz inv}-\eqref{eq: traceless}.
A similar approach was taken in \cite{Zhiboedov}, in the case where $s_i-t_i=1$.
The success of this approach relied on the ability to model the action of $\partial_{P_i}\cdot D_{Z_i}$ on three-point functions with $\Delta_i=d+s_i-2$ in terms of a differential operator written using the building blocks $H_{ij}$ and $V_{i,jk}$. This leads to the question: when can the action of $(\partial_{P_i}\cdot D_{Z_i})^{s_i - t_i}$ on $n$-point functions be expressed entirely in terms of the building blocks $P_{ij}$, $H_{ij}$ and $V_{i,jk}$?

\medskip

In this section, we reproduce the construction of \cite{Zhiboedov}, explain why it works and suggest its generalisation to arbitrary powers of $\partial_{P_i}\cdot D_{Z_i}$ for $n=3$. 
Note that, as will be made clear below, writing $(\partial_{P_i}\cdot D_{Z_i})^{s_i-t_i}$ in terms of the building blocks is not the same as composing the differential operator $\partial_{P_i}\cdot D_{Z_i}$ written in terms of the building blocks $s_i-t_i$ times. 
This is because the building blocks version of $\partial_{P_i}\cdot D_{Z_i}$ is not simply a coordinate transformation of the original differential operator.
In a sense, it is a differential operator that acts in the same way as $\partial_{P_i}\cdot D_{Z_i}$, when restricted to a specific vector subspace. When thought of as differential operators acting on the larger vector space of all three-point functions, their action on the same function can produce different results.
We will mainly consider three-point functions, and describe when and how our results can be applied to higher-point functions. 

\medskip

As in Section \ref{sec:matching_pols}, in what follows we will be thinking of the building blocks as formal variables. As described in Section~\ref{sec:ConfNpoint}, a three-point function~$G_{\mathbf{s},\mathbf{\Delta}}$ can be thought of as an element of the $\mathbb{C}$-vector space
\begin{equation}\label{eqn: 3point vec space}
    P^{-\tau_{12,3}}_{12}P^{-\tau_{23,1}}_{23}P^{-\tau_{31,2}}_{13}\cdot W_{\mathbf{s}}\,,
\end{equation}
where we have defined $\tau_{ij,k}:=\frac{1}{2}(s_i+s_j-s_k + \Delta_i + \Delta_j -\Delta_k)$, and
\begin{equation}
        W_{\mathbf{s}}:=\mathrm{span}_{\mathbb{C}}\left\{H^{h_{1}}_{12}H^{h_{2}}_{23}H^{h_{3}}_{13}V^{v_1}_{1,23}V^{v_2}_{2,31}V^{v_3}_{3,12}\quad:\quad \begin{aligned} v_1+h_1+h_3&=s_1\\v_2+h_1+h_2&=s_2\\v_3+h_2+h_3&=s_3\end{aligned}\right\}.
\end{equation}
To permit the action of the differential operators $(\partial_{P_i}\cdot D_{Z_i})^{s_i-t_i}$ on this vector space, we introduce a linear map
\begin{equation}
    f:\mathbb{C}(P_{ij},H_{ij},V_{ijk})\rightarrow\mathbb{C}(P_i,Z_i)\,.
\end{equation}
Choosing to work with the basis for \eqref{eqn: 3point vec space} given by the Laurent monomials
\begin{equation}\label{eqn: 3point monomial}
\frac{H^{h_{1}}_{12}H^{h_{2}}_{23}H^{h_{3}}_{13}V^{v_1}_{1,23}V^{v_2}_{2,31}V^{v_3}_{3,12}}{P^{\tau_{12,3}}_{12}P^{\tau_{23,1}}_{23}P^{\tau_{31,2}}_{13}}\,,
\end{equation}
we define $f$ to be the evaluation map that replaces each of the formal variables $P_{ij}$, $H_{ij}$ and $V_{i,jk}$ in the monomial \eqref{eqn: 3point monomial} with the embedding-space structures:
\begin{align}
      P_{ij}
  & \ \mapsto  \ P_i\cdot P_j\,,\\
  H_{ij}
  & \ \mapsto \ -2\bigl[(Z_i\cdot Z_j)(P_i\cdot P_j)
             -(Z_i\cdot P_j)(Z_j\cdot P_i)\bigr]\,,\\
  V_{ijk}
  & \ \mapsto \ \frac{(Z_i\cdot P_j)(P_i\cdot P_k)
           -(Z_i\cdot P_k)(P_i\cdot P_j)}
          {P_j\cdot P_k}\,.
\end{align}
The differential operators $(\partial_{P_i}\cdot D_{Z_i})^{s_i-t_i}$ can then act on the $\mathbb{C}$-vector space
\begin{equation}
    f(P^{-\tau_{12,3}}_{12}P^{-\tau_{23,1}}_{23}P^{-\tau_{31,2}}_{13}\cdot W_{\mathbf{s}})\,.
\end{equation}

\subsection{Existence of the lifted operator}
We begin with the following proposition.

\begin{proposition}\label{prop: diff operator 1}
    Fix $i\in\{1,2,3\}$. Allow for any $\mathbf{s}\in\mathbb{Z}_{\geq 0}^3$ with $s_i>0$ and $\Delta_{j}\in\mathbb{C}$ for $j\neq i$. Set $t_i=s_i-1$, $s_i\ge1$. Then the differential operator $(\partial_{P_i}\cdot D_{Z_i})^{s_i-t_i}$ is a map    \begin{equation}\label{eqn: preserve Pij}
    (\partial_{P_i}\cdot D_{Z_i})^{s_i-t_i}:f(P^{-\tau_{12,3}}_{12}P^{-\tau_{23,1}}_{23}P^{-\tau_{31,2}}_{13}\cdot W_{\mathbf{s}})\rightarrow f(P^{-\tau_{12,3}}_{12}P^{-\tau_{23,1}}_{23}P^{-\tau_{31,2}}_{13}\cdot W_{\mathbf{s}-(s_i-t_i)\mathbf{e}_i})\,,
\end{equation}
    if and only if $\Delta_i=d-1+t_i$. 
\end{proposition}

\begin{remark}
    Although not proven symbolically due to being computationally very demanding, the same statement can easily be verified to hold for $s_i - t_i = 2$ for a vast number of choices of $s_1$, $s_2$ and $s_3$ (see \cite{zenodo}). In particular, we believe that the result for $s_i-t_i=1$ extends to arbitrary $s_i-t_i$. This is made precise in Conjecture~\ref{conj: 3point diff op conjecture}.
\end{remark}

\begin{proof}
The action of the differential operators $\partial_{P_i}\cdot D_{Z_i}$ on
\begin{equation}\label{eqn: embedding space monomial}
    f\left(\frac{H^{h_{1}}_{12}H^{h_{2}}_{23}H^{h_{3}}_{13}V^{v_1}_{1,23}V^{v_2}_{2,31}V^{v_3}_{3,12}}{P^{\tau_{12,3}}_{12}P^{\tau_{23,1}}_{23}P^{\tau_{31,2}}_{13}}\right),
\end{equation}
can be computed explicitly. In \texttt{Mathematica} \cite{zenodo}, we show that the results can be rewritten in terms of the building blocks by reversing the action of $f$ when $\Delta_i=d-2+s_i$. In particular, we explicitly perform the rewriting in the cases where we have the right scaling dimension, and show that for other scaling dimensions, there are leftover terms containing factors of $\, P_i \cdot Z_j\,$. For generic values of $\mathbf{s}$, these terms cannot be rewritten in terms of the building blocks. This can be shown by differentiating the expression with respect to these residual embedding-space inner products, and noticing that this is written solely in terms of embedding-space structures and can only be vanishing if $\Delta_i$ takes the value specified above.

\smallskip

The explicit computation also reveals that the powers of the overall factors of $P_i\cdot P_j$ are invariant and the spins are reduced by one.
In other words, for $\Delta_i=d-2+s_i$ and $s_i-t_i=1$, we find \eqref{eqn: preserve Pij}. The reduction in spin follows from the fact that elements of the vector space \eqref{eqn: 3point vec space} are polynomials in the $Z_i$. From the form of $\partial_{P_i}\cdot D_{Z_i}$ it is clear that the action of $\partial_{P_i}\cdot D_{Z_i}$ on polynomials in the $Z_i$ reduces their degree in the $Z_i$ by one.
\end{proof}

\begin{remark}
Note that this feature of preserving the overall factor in the $P_{ij}$ is not generic. For $\Delta_i$ differing from the value given above, the differential operator $\partial_{P_i}\cdot D_{Z_i}$
does not preserve the overall factor in the $P_{ij}$. In addition, perhaps counter-intuitively, the overall factor of $P^{-\tau_{12,3}}_{12}P^{-\tau_{23,1}}_{23}P^{-\tau_{31,2}}_{13}$ is crucial in Proposition \ref{prop: diff operator 1}. Indeed, via direct computation one can verify that the differential operator $(\partial_{P_i}\cdot D_{Z_i})^{s_i-t_i}$ does \emph{not} map $f(W_{\mathbf{s}})$ to $f(W_{\mathbf{s}-(s_i-t_i)\mathbf{e}_i})$. To see how this could be the case, consider the example where $s_i-t_i=1$. In this case, we have
\begin{equation}\label{eqn: 1st order action}
    {\small\partial_{P_i}\cdot D_{Z_i}\Bigg(f\Bigg(\frac{H^{h_{1}}_{12}H^{h_{2}}_{23}H^{h_{3}}_{13}V^{v_1}_{1,23}V^{v_2}_{2,31}V^{v_3}_{3,12}}{P^{\tau_{12,3}}_{12}P^{\tau_{23,1}}_{23}P^{\tau_{31,2}}_{13}}\Bigg)\Bigg)=\partial_{P_i}\cdot\Bigg(\frac{D_{Z_i}(f(H^{h_{1}}_{12}H^{h_{2}}_{23}H^{h_{3}}_{13}V^{v_1}_{1,23}V^{v_2}_{2,31}V^{v_3}_{3,12})}{f(P^{\tau_{12,3}}_{12}P^{\tau_{23,1}}_{23}P^{\tau_{31,2}}_{13})}\Bigg).}
\end{equation}
However,
\begin{equation*}
    \partial_{P_i}(f(P^{-\tau_{12,3}}_{12}P^{-\tau_{23,1}}_{23}P^{-\tau_{31,2}}_{13}))\neq \mathbf{0}\,,
\end{equation*}
and so~\eqref{eqn: 1st order action} is not equal to
\begin{equation*}
    \frac{\partial_{P_i}\cdot D_{Z_i}(f(H^{h_{1}}_{12}H^{h_{2}}_{23}H^{h_{3}}_{13}V^{v_1}_{1,23}V^{v_2}_{2,31}V^{v_3}_{3,12}))}{f(P^{\tau_{12,3}}_{12}P^{\tau_{23,1}}_{23}P^{\tau_{31,2}}_{13})}\,.
\end{equation*}
The existence of the overall factor in the $P_{ij}$, with the right dependence on $s_i$ and $t_i$, is essential for Proposition \ref{prop: diff operator 1} to hold.
\end{remark}

The expression \eqref{eqn: preserve Pij} is suggestive of the fact that, for the right value of the scaling dimension, the action of the differential operator $\partial_{P_i}\cdot D_{Z_i}$
can be modelled by the action of a differential operator in the $H_{ij}$ and $V_{i,jk}$ that acts solely on elements of $W_{\mathbf{s}}$. Indeed, it implies the following theorem.

\begin{theorem}\label{thm: diff op 1}
Fix $i\in\{1,2,3\}$. Allow for any $\mathbf{s}\in\mathbb{Z}_{\geq 0}^3$ and $\Delta_{j}\in\mathbb{C}$ for $j\neq i$. Set $t_i=s_i-1$, $s_i\ge1$. Then there exists a differential operator $\mathcal{D}_{i,s_i-t_i}$ acting on $\mathbb{C}[H_{ij},V_{i,jk}]$ satisfying
\begin{equation*}
(\partial_{P_i}\cdot D_{Z_i})^{s_i-t_i}\left(f\left(\frac{H^{h_{1}}_{12}H^{h_{2}}_{23}H^{h_{3}}_{13}V^{v_1}_{1,23}V^{v_2}_{2,31}V^{v_3}_{3,12}}{P^{\tau_{12,3}}_{12}P^{\tau_{23,1}}_{23}P^{\tau_{31,2}}_{13}}\right)\right)=f\left(\frac{\mathcal{D}_{i,s_i-t_i}(H^{h_1}_{12}H^{h_2}_{23}H^{h_3}_{13}V^{v_1}_{1,23}V^{v_2}_{2,31}V^{v_3}_{3,12})}{P^{\tau_{12,3}}_{12}P^{\tau_{23,1}}_{23}P^{\tau_{31,2}}_{13}}\right),
\end{equation*}
if and only if $\Delta_i=d-1+t_i$, where the $h_i$ and $v_i$ must satisfy
\begin{equation}
    v_1+h_1+h_3=s_1\,,\quad v_2+h_1+h_2=s_2\,,\quad 
    v_3+h_2+h_3=s_3\,.
\end{equation}
\end{theorem}

\begin{remark}
    We note that Theorem \ref{thm: diff op 1} is exactly what one would expect from the following physicist-oriented argument. As shown in \cite[p.~9-10]{Osborn}, the derivative of a bosonic conformal primary of spin $s\ge1$ and scaling dimension $\Delta$ is a conformal primary iff $\Delta=d-2+s$. At the level of $n$-point functions, this means that the action of $\partial_{P_i}\cdot D_{Z_i}$ on an $n$-point function of conformal primaries is only an $n$-point function of conformal primaries when $\Delta_i=d-2+s_i$. This means that the action of $\partial_{P_i}\cdot D_{Z_i}$ on such an $n$-point function can only be written in terms of the embedding-space structures $P_{ij},\, H_{ij}$ and $V_{i,jk}$ when $\Delta_i=d-2+s_i$. 
    Repeating the argument of \cite{Osborn} for higher order derivatives, one can see how this statement could be generalised to cases where $s_i-t_i>1$.
\end{remark}
Based on the previous remark, we make the following conjecture.
\begin{conjecture}\label{conj: higher order derivative}
    When acting on $n$-point functions $G_{\mathbf{s},\mathbf{\Delta}}$, the differential operators ${(\partial_{P_i}\cdot D_{Z_i})^{s_i-t_i}}$ can be written entirely in terms of the $H_{ij}$, $V_{i,jk}$ and the cross-ratios $u_{ijkl}$ if and only if $$\Delta_i=d-1+t_i.$$
\end{conjecture}
In the case of three-point functions we can formulate this conjecture more precisely as an extension of Theorem \ref{thm: diff op 1}.
\begin{conjecture}\label{conj: 3point diff op conjecture}
    Fix $i\in\{1,2,3\}$. Let $\mathbf{s}\in\mathbb{Z}_{\geq 0}^3$, $s_i>0$, $\Delta_{j}\in\mathbb{C}$ for $j\neq i$ and depth $t_i\in\{0,\ldots,s_i-1\}$. There exists a differential operator $\mathcal{D}_{i,s_i-t_i}$ acting on $\mathbb{C}[H_{ij},V_{i,jk}]$ satisfying
    \begin{equation*}
    (\partial_{P_i}\cdot D_{Z_i})^{s_i-t_i}\left(f\left(\frac{H^{h_{1}}_{12}H^{h_{2}}_{23}H^{h_{3}}_{13}V^{v_1}_{1,23}V^{v_2}_{2,31}V^{v_3}_{3,12}}{P^{\tau_{12,3}}_{12}P^{\tau_{23,1}}_{23}P^{\tau_{31,2}}_{13}}\right)\right)=f\left(\frac{\mathcal{D}_{i,s_i-t_i}(H^{h_1}_{12}H^{h_2}_{23}H^{h_3}_{13}V^{v_1}_{1,23}V^{v_2}_{2,31}V^{v_3}_{3,12})}{P^{\tau_{12,3}}_{12}P^{\tau_{23,1}}_{23}P^{\tau_{31,2}}_{13}}\right),
    \end{equation*}
    if and only if $\Delta_i=d-1+t_i$, where the $h_i$ and $v_i$ must satisfy
    \begin{equation*}
        v_1+h_1+h_3=s_1\,,\quad
        v_2+h_1+h_2=s_2\,,\quad 
        v_3+h_2+h_3=s_3\,.
    \end{equation*}
\end{conjecture}

\begin{remark}
    Note that our results in this section consider only the existence of the operator $\mathcal{D}_{i,s_i-t_i}$, not its uniqueness. We do not know whether the operator  $\mathcal{D}_{i,s_i-t_i}$ is unique for all choices of $s_i-t_i$. For general $d$, our computation that we describe below shows that it is unique in the case $s_i-t_i=1$ and suggests it is not unique already for $s_i-t_i=2$ in $d = 3$. 
\end{remark}

\subsection{Construction and explicit form}
We now present a computation that allows one to construct the differential operators $\mathcal{D}_{i,s_i-t_i}$.
This computation treats $s_1$, $s_2$, $s_3$, $\Delta_2$ and $\Delta_3$ as parameters. 
\newline

\noindent\textbf{Input:} A choice of $i\in\{1,2,3\}$, $s_i-t_i\in\mathbb{N}$ and $d\ge3$.

\begin{enumerate}
\item Act with the differential operator $(\partial_{P_i}\cdot D_{Z_i})^{s_i-t_i}$ on the monomial
\begin{equation*}
    \frac{H^{h_{1}}_{12}H^{h_{2}}_{23}H^{h_{3}}_{13}V^{v_1}_{1,23}V^{v_2}_{2,31}V^{v_3}_{3,12}}{P^{\tau_{12,3}}_{12}P^{\tau_{23,1}}_{23}P^{\tau_{31,2}}_{13}}\,,
\end{equation*}
where we have set
\begin{align*}
    \Delta_i=d-1+t_i, \quad v_1=s_1-h_{1}-h_{3}, \quad v_2=s_2-h_1-h_2, \quad v_3=s_3-h_3-h_2\,.
\end{align*}
Here $h_1$, $h_2$ and $h_3$ should be treated as parameters.
\item Write the result of this action in terms of the building blocks by reversing the action of~$f$. That this is possible is guaranteed by Proposition \ref{prop: diff operator 1} when $s_i-t_i=1$ for any $d\ge3$ and is expected to hold for other choices of $s_i-t_i$. Then multiply the entire expression by $P^{\tau_{12,3}}_{12}P^{\tau_{23,1}}_{23}P^{\tau_{31,2}}_{13}$
to cancel the overall factor in the $P_{ij}$. Call the result $X$.
\item Construct the most general linear differential operator $\mathcal{D}_{i,s_i-t_i}$ of order $3(s_i-t_i)$ in variables $H_{ij}$, $V_{i,jk}$, without partial derivatives with respect to $H_{jk}$ with neither of $j,k$ equal to $i$. That such partial derivatives should not appear in a differential operator modelling $(\partial_{P_i}\cdot D_{Z_i})^{s_i-t_i}$ is clear from the form of $\partial_{P_i}\cdot D_{Z_i}$ in terms of the embedding-space vectors. 
Leave placeholders for the coefficients of the differentials. These coefficients should be understood as being elements of the polynomial ring $\mathbb{C}[H_{ij},V_{i,jk}]$.
    
\item Apply $\mathcal{D}_{i,s_i-t_i}$ to the monomial $H^{h_1}_{12}H^{h_2}_{23}H^{h_3}_{13}V^{v_1}_{1,23}V^{v_2}_{2,31}V^{v_3}_{3,12}$. Call the result $Y$.

\item Fix the coefficients in $\mathcal{D}_{i,s_i-t_i}$ by comparing $X$ and $Y$. 
This can be done by evaluating the expression $X=Y$ at a sufficiently large number of choices of values for the $s_j$ and $\Delta_{j}$, $j\neq i$ with appropriate $h_j$, such that $s_i \geq h_{i} + h_{i-1}$. 
Eventually one arrives at a set of equations linear in the free coefficients in the differential operator ansatz $\mathcal{D}_{i, s_i-t_i}$ that can be solved in terms of $H_{ij}$, $V_{i,jk}$, $\Delta_2$ and $\Delta_3$. 
It could occur that not all the coefficients are fixed by this process. 
If after a large number of iterations no new coefficient is fixed, we assume that the unfixed coefficients are free and thus there are multiple operators $\mathcal{D}_{i,s_i-t_i}$. 
To make one choice of the operator, we set these to zero. 
The chosen operator is guaranteed to act correctly on the spaces $W_\mathbf{s}$ whose bases we used in this computation to obtain $\mathcal{D}_{i,s_i-t_i}$ and we conjecture it acts correctly on all~$W_\mathbf{s}$. 
\end{enumerate}

\noindent\textbf{Output:} The differential operator $\mathcal{D}_{i,s_i-t_i}$.\newline

Performing this computation for $i=1$ and $s_i-t_i=1$, we obtain a unique operator $\mathcal{D}_{i,1}$ and reproduce the result of \cite{Zhiboedov}.
Note that the author of \cite{Zhiboedov} uses a slightly different definition of the building blocks.
Thus, at face value the differential operator presented here looks different to that in \cite[Appendix A]{Zhiboedov}.

\begin{proposition}
    The differential operator $\mathcal{D}_{1,1}$ for arbitrary $d$ is given by
    \begin{equation}
    \mathcal{D}_{1,1} \;=\; \frac{d-2}{2}\,\mathcal{D}_{1,1}^{(1)} + \tfrac{1}{2}\,\mathcal{D}_{1,1}^{(2)} + \mathcal{D}_{1,1}^{(3)},
\end{equation}
where the three components are given explicitly below.
\begin{align}
\mathcal{D}^{(1)}_{1,1} &=
    \frac{1}{2}(\Delta_3 - \Delta_2)\,\partial_{V_1}
    + (d - 1 + \Delta_3 - \Delta_2)\,V_2 \,\partial_{H_{12}}
    - (2\leftrightarrow 3),\\[8pt]
\mathcal{D}_{1,1}^{(2)} &= 
    \frac{1}{2}(\Delta_2 - \Delta_3) V_1\,\partial_{V_1}^2 + (3d - 2 - 2\Delta_2 + 2\Delta_3)\, H_{12} V_2\,\partial_{H_{12}}^2 \notag\\
&+ 2\Big[(\Delta_2 - \Delta_3 - 1)\, H_{12} + (d-2)\,V_1 V_2\Big]\partial_{V_1}\partial_{H_{12}} \notag\\[3pt]
&+ (d-2)V_2^2\,\partial_{V_2}\partial_{H_{12}}
 - (d-2)(H_{23} + V_2 V_3)\,\partial_{V_2}\partial_{H_{31}} \notag\\[3pt]
&+ \frac{1}{2}\Big[
    3(d-2)H_{12}V_3
    - (\Delta_2 - \Delta_3)(H_{23}V_1  + 2 \,H_{12}V_3 + 2V_1 V_2 V_3)
  \Big]\partial_{H_{12}}\partial_{H_{31}} \notag\\
& - (2\leftrightarrow3), \notag \\[8pt]
\mathcal{D}_{1,1}^{(3)} &= 
    \frac{1}{2}(H_{12} + V_1 V_2)\,\partial_{V_1}^2\partial_{V_2} - \frac{1}{2}H_{12} V_1\,\partial_{V_1}^2\partial_{H_{12}} \notag\\
&- H_{12} V_2\,\partial_{V_1}\partial_{V_2}\partial_{H_{12}}
  - (H_{23}V_1 + H_{12}V_3 + 2V_1 V_2 V_3)\,\partial_{V_1}\partial_{V_2}\partial_{H_{31}} \notag\\[3pt]
&+ 2H_{12} V_1 V_2\,\partial_{V_1}\partial_{H_{12}}^2 
 + H_{12}^2 V_2\,\partial_{H_{12}}^3 \notag\\[3pt]
&  - (H_{12}H_{23}V_1 - 2H_{12}H_{31}V_2 + H_{12}^2 V_3 + 2H_{12}V_1 V_2 V_3)\,\partial_{H_{12}}^2\partial_{H_{31}} \notag\\[3pt]
&+ H_{12} V_2^2\,\partial_{V_2}\partial_{H_{12}}^2 + (H_{12}H_{23} + H_{12}V_2 V_3)\,\partial_{V_3}\partial_{H_{12}}^2\notag \\[3pt]
& + (H_{23}V_1 V_2 + H_{31}V_2^2 + H_{12}V_2 V_3 + 2V_1 V_2^2 V_3)\,\partial_{V_2}\partial_{H_{12}}\partial_{H_{31}} \notag\\[3pt]
& - (2\leftrightarrow 3), \notag
\end{align}
where we have adopted the shorthand $V_1:=V_{1,23}$, $V_2:=V_{2,31}$ and $V_3:=V_{3,12}$, and where $(2\leftrightarrow 3)$ denotes the image of the preceding terms under the transposition of labels 2 and 3. The differential operators $\mathcal{D}_{2,1}$ and $\mathcal{D}_{3,1}$ can be obtained from $\mathcal{D}_{1,1}$ through a simple relabelling.
\end{proposition}

We have similarly computed the differential operator $\mathcal{D}_{1,2}$ in $d=3$. In this case one coefficient was left free, and could therefore be set to zero. The output can be found at \cite{zenodo}. 

\begin{remark}
    Note that the differential operator $\mathcal{D}_{1,1}$ is independent of the spins $s_1$, $s_2$ and $s_3$. Note also that its only dependence on $\Delta_2$ and $\Delta_3$ is through their difference. Both of these facts also hold for the differential operator $\mathcal{D}_{1,2}$ in $d=3$. The lack of dependence on the spins suggests the following: to construct the differential operators $\mathcal{D}_{i,s_i-t_i}$ for general spins $\mathbf{s}$, one only needs to construct their action on $W_{\mathbf{s}}$ for any single choice of $s_1$, $s_2$ and $s_3$. In particular, it suggests that the differential operator $\mathcal{D}_{1,2}$ that we derived for specific values of the spins in \texttt{Mathematica}, holds for general spins. It would be interesting to obtain both mathematical proof and physical explanation of this. 
\end{remark}

\section{Conclusions and outlook}

In this paper we have developed methods to constrain the set of independent tensor structures appearing in conformal correlators from a purely algebraic and combinatorial approach. Working in terms of the building blocks $P_{ij}$, $H_{ij}$, and $V_{i,jk}$, we provide concrete and algorithmic methods for enumerating and generating these structures, taking Bose symmetry, partial conservation and algebraic relations into consideration. In doing this, we provide rigorous proofs for results used widely in the physics literature. As we focus on the algebraic perspective, we attack the problem from an angle that has largely been avoided in the literature, where representation-theoretic approaches have been used instead, see e.g.~\cite{DKKPS, KSD}.

\medskip

We now summarise the open questions and conjectures that appeared in this work. 
\begin{enumerate}
    \item In recent years, polytopes appearing in physics have received substantial attention from the discrete geometry community. A primary example is that of cosmological polytopes \cite{cosmopolytopes}. In the same spirit, we find that it would be interesting to investigate the mathematical properties of conformal $n$-point polytopes $\mathcal{C}_n$ in more detail.

    \item Section \ref{per: com} connects our counting problem to Kostka numbers and Littlewood-Richardson coefficients. This raises the question of whether there is a natural representation theoretic interpretation of our combinatorial results that connects to~\cite{KSD}.

    \item In Section \ref{sec: alg rel} we studied algebraic relations between the basic building blocks \eqref{eq: confblocks intro}. From the physics perspective, one wishes to understand the number of monomials in $H_{ij}$ and $\mathcal{V}_{ij}$ of a given multidegree that are algebraically independent over the field $\mathbb{C}(\mathbf{u})$ of rational functions in the cross-ratios. This is encoded by the Hilbert function $h_I$ introduced in the end of Section \ref{sec: alg rel}. While we give an upper bound in the form of the Hilbert function $h_{\Tilde{I}}$, it is an open problem to compute the function $h_I$ itself. This is computationally challenging because both the elements of the coefficient field and the blocks $H_{ij}, \mathcal{V}_{ij}$ depend on the same vectors $P_i$, restricting the use of numerical methods.  

    \item In the context of Section \ref{sec: Bose}, a challenging problem is to computationally obtain a basis of $n$-point structures respecting Bose symmetry for $n\geq 4$. The difficulty here again comes from the fact that these structures depend on cross-ratios arbitrarily.

    \item It would be interesting to obtain an explanation (physical and mathematical) of the fact that the operator $\mathcal{D}_{i,s_i-t_i}$ constructed in Section \ref{sec: conservation} does not depend on the values of spins and only depends on the differences between scaling dimensions.

   \item The approach of \cite{Zhiboedov} yields generating functionals for all three-point function structures of fully conserved currents ($s_i - t_i = 1$) of arbitrary spin in arbitrary $d$. It remains an open problem to extend this to partially conserved currents of arbitrary depth, in other words to obtain generating functionals for the general solution to
    \begin{equation}
     \mathcal{D}_{i,s_i-t_i}\, g_{\mathbf{s},\mathbf{\Delta}}(H_{ij}, V_i) = 0\,,
    \end{equation}
    where $G_{\mathbf{s},\mathbf{\Delta}} = g_{\mathbf{s},\mathbf{\Delta}}(H_{ij}, V_i) \big/ P^{\tau_{12,3}}_{12}P^{\tau_{23,1}}_{23}P^{\tau_{31,2}}_{13}$. In Section~\ref{sec: conservation} we established the conditions under which $\mathcal{D}_{i,s_i-t_i}$ exists as a differential operator on the $H_{ij}$ and $V_{i,jk}$, and provided an algorithm for computing it explicitly. A remaining open problem is to solve the resulting differential equation. Beyond obtaining the generating functionals themselves, one could, in the spirit of \cite{Zhiboedov}, investigate whether the resulting structures admit an interpretation in terms of known CFTs involving partially conserved currents like \cite{Brust_2017}, or whether there exist structures that are not generated by any known CFTs at all.
\end{enumerate}

\medskip

We conclude by noting that our methods have immediate applications to the conformal and cosmological bootstrap programs, providing efficient means to construct and understand the algebraic structures underlying conformal correlators.
These methods can also be applied in other contexts. Indeed, integer solutions of \cite[Eq.~15]{Heslop2026} are lattice points of the transportation polytopes introduced in Remark \ref{rem: transportation polytope}, and the generating function \cite[Eq.~17]{Heslop2026} can be seen as the Hilbert series of a polynomial ring. More generally, Hilbert series have already been used as generating functions in effective field theory and standard-model physics \cite{HLMM,LM1,LM2,YZ}. Through this paper, we hope to invite physicists to further explore applications of existing, powerful results in nonlinear algebra, combinatorics and invariant theory.

\section*{Acknowledgements}
We thank Ben Hollering and Svala Sverrisdóttir for providing code to compute relations among building blocks, Thomas Kahle for pointing out the \texttt{barvinok} package, Daniel Baumann, Yassine El Maazouz, Callum R.~T.~Jones, and Bernd Sturmfels for helpful discussions.
\smallskip

Our research is funded by the European Union (ERC, UNIVERSE PLUS, 101118787). Views and opinions expressed are, however, those of the authors only and do not necessarily reflect those of the European Union or the European Research Council Executive Agency. Neither the European Union nor the granting authority can be held responsible for them. 

\printbibliography[heading=bibintoc]

\bigskip \medskip \bigskip

\noindent
\small {\bf Authors' addresses:}
\smallskip

\noindent Viktoriia Borovik, MPI-MiS, Leipzig, Germany
\hfill {\tt viktoriia.borovik@mis.mpg.de}

\noindent Claire de Korte, MPI-MiS, Leipzig, Germany
\hfill {\tt claire.dekorte@mis.mpg.de}

\noindent Nathan Meurrens, University of Amsterdam, the Netherlands
\hfill {\tt n.meurrens@uva.nl}

\noindent Dmitrii Pavlov, MPI for Physics, Garching, Germany
\hfill {\tt pavlov@mpp.mpg.de}
\end{document}